\documentclass[12pt,reqno]{amsart}
\usepackage[utf8]{inputenc}
\usepackage[T1]{fontenc}
\usepackage[usenames, dvipsnames]{color}
\usepackage{ulem}

\usepackage{dsfont, amsfonts, amsmath, amssymb,amscd, stmaryrd, latexsym, amsthm, dsfont}
\usepackage[frenchb,english]{babel}
\usepackage{enumerate}
\usepackage{longtable}
\usepackage{geometry}
\usepackage{float}
\usepackage{tikz}
\usetikzlibrary{shapes,arrows}
\usepackage{float}
\usepackage{tikz}
\usepackage{xypic}
\usetikzlibrary{shapes,arrows}

\usepackage{pifont}
\usepackage{float}
\usepackage{tikz}
\usepackage{xypic}
\usetikzlibrary{shapes,arrows}

\newcommand{\rsp}{\raisebox{0em}[2.7ex][1.3ex]{\rule{0em}{2ex} }}
\newtheorem{theorem}{Theorem}[section]
\newtheorem*{maintheorem}{Theorem}
\newtheorem*{mainproposition}{Proposition}
\newtheorem*{maincorollary}{Corollary}
\newtheorem*{maindefinition}{Definition}

\newtheorem{lemma}[theorem]{Lemma}

\newtheorem{corollary}[theorem]{Corollary}

\theoremstyle{remark}
\newtheorem{remark}[theorem]{\bf Remark}

\newtheorem{example}[theorem]{\bf Example}

\usepackage[pagebackref]{hyperref}
\renewcommand*{\backref}[1]{}\renewcommand*{\backrefalt}[4]{\ifcase #1 (\tt not cited)\or (\tt cited on page~#2)\else (\tt cited on pages~#2)\fi}

\usepackage{tikz-cd}
\usepackage{hyperref}
\hypersetup{
	colorlinks=true,
	urlcolor=blue,
	citecolor=blue,linkcolor=blue}
\def\NN{\mathbb{N}}
\def\RR{\mathds{R}}
\def\HH{I\!\! H}
\def\QQ{\mathbb{Q}}
\def\CC{\mathds{C}}
\def\ZZ{\mathbb{Z}}
\def\DD{\mathds{D}}
\def\OO{\mathcal{O}}
\def\kk{\mathds{k}}
\def\KK{\mathbb{K}}
\def\ho{\mathcal{H}_0^{\frac{h(d)}{2}}}
\def\LL{\mathbb{L}}
\def\L{\mathds{k}_2^{(2)}}
\def\M{\mathds{k}_2^{(1)}}
\def\k{\mathds{k}^{(*)}}
\def\l{\mathds{L}}
 
\usepackage{multirow}

\def\kk{\mathds{k}}
\begin{document}
	
	\def\NN{\mathbb{N}}
	\def\RR{\mathds{R}}
	\def\HH{I\!\! H}
	\def\QQ{\mathbb{Q}}
	\def\CC{\mathds{C}}
	\def\ZZ{\mathbb{Z}}
	\def\DD{\mathds{D}}
	\def\OO{\mathcal{O}}
	\def\kk{\mathds{k}}
	\def\KK{\mathbb{K}}
	\def\ho{\mathcal{H}_0^{\frac{h(d)}{2}}}
	\def\LL{\mathbb{L}}
	\def\L{\mathds{k}_2^{(2)}}
	\def\M{\mathds{k}_2^{(1)}}
	\def\k{\mathds{k}^{(*)}}
	\def\l{\mathds{L}}
	\def\2r{\mathrm{rank}}
	\def\rg{\mathrm{rank}}
	\def\C{\mathrm{C}}
	\def\Y{\mathbf{Y}}
	\def\V{\mathbf{V}}
	\def\vep{\varepsilon}

	\selectlanguage{english}
	
	
	\title[The Maximal Unramified Pro-$2$-Extension over the   $\mathbb{Z}_2$-Extension]{On the Existence of the Maximal Unramified Pro-$2$-Extension over the Cyclotomic $\mathbb{Z}_2$-Extension with Prescribed Metacyclic Galois Group}
	
	
	\author[M. M. Chems-Eddin]{Mohamed Mahmoud Chems-Eddin}
	\address{Mohamed Mahmoud CHEMS-EDDIN: Department of Mathematics, Faculty of Sciences Dhar El Mahraz, Sidi Mohamed Ben Abdellah University, Fez,  Morocco}
	\email{2m.chemseddin@gmail.com}
	
	\author[H. El Mamry]{Hamza El Mamry}
	\address{Hamza EL MAMRY: Departement of Mathematics, Faculty of Sciences Dhar El Mahraz,  Sidi Mohamed Ben Abdellah University, Fez, Morocco}
	\email{Hamza.elmamry@usmba.ac.ma}


	\begin{abstract}
		For an     integer $m\geq 2$, we     aim   to investigate the realizability of types of metacyclic-nonmodular  groups,     whose  abelianization is $\ZZ/2 \ZZ\times\ZZ/2^m \ZZ$,  as the Galois group of   the maximal unramified  $2$-extension (resp. pro-$2$-extension) over certain number fields of $2$-power degree (resp. cyclotomic $\mathbb Z_2$-extensions).  Furthermore, we present some new techniques for studying Greenberg's conjecture for some  number fields. 
		In particular, the reader can find results concerning the real quadratic fields $F=\QQ(\sqrt{\eta q rs})$, the real biquadratic fields $K=\QQ(\sqrt{\eta q},\sqrt{rs})$, with $\eta\in\{1,2\}$, and the Fröhlich  multiquadratic   fields   of the form $\mathbb F=\QQ(\sqrt{q }, \sqrt {r},  \sqrt{s})$, where  $q$, $r$ and $s$ are  odd prime  numbers.
	\end{abstract}

	\selectlanguage{english}
	
	\maketitle

	\section{\bf  Introduction}

	Let $n\geq 1$ be a natural number and $G$ be a finite metabelian $2$-group of rank $2$ such that $G^{ab}:=G/G'  \simeq\ZZ/2 \ZZ\times\ZZ/2^n \ZZ$,  where $G':=[G,G]$ is the
	derived subgroup of $G$ and so $G^{ab}$ is the abelianization of   $G$. 
	The study of the realizability of these $2$-groups as the   Galois group of the maximal  unramified extension of quadratic fields was developed  by Kisilevsky,  Benjamin and     Snyder (cf. \cite{Ben17,Ben2006,Ben1999,benjashnepreprint1993,Ki76}).
	
	\bigskip
	
	For $n=1$, it is known that $G$ is in one of the following structures: Klein four-group,      the quaternion,  dihedral or semidihedral groups (cf. \cite{Ki76}).
	The study of the realizability of these $2$-groups  as the   Galois group of the maximal unramified pro-$2$-extension of certain cyclotomic $\mathbb Z_2$-extensions  was motivated by Mizusawa in his thesis (cf. \cite{MizusawaThesis}). Subsequently, this has been extensively investigated by several authors in previous years (cf. \cite{MizusawaAnnMath2014,Ajn,MohibMathNachr2016} and references cited therein). 
	 	For the reader's  convenience, recall that a   $\ZZ_2$-extension of a number field $k$ is an infinite extension of $k$ denoted by  $k_\infty$ such that 
	 	$\mathrm{Gal}(k_\infty/k)$ is topologically isomorphic to the additive group of $ \ZZ_2$, where $\ZZ_2$ is the ring of $2$-adic integers. For each $s\geq 1$, the extension $k_\infty/k$ contains a unique field    of degree $2^s$ denoted by $k_s$ and called the $s$th layer  of the   $\ZZ_2$-extension of $k$. Furthermore, we have:
	 $k=k_0\subset k_1 \subset k_2 \subset\cdots \subset k_s \subset \cdots  \subset k_\infty=\bigcup_{s\geq 0} k_s.$ 
	In particular, if $  \QQ_{2, s}$ denote the field $\QQ(2\cos( {2\pi}/{2^{s+2}}))$, for all $s\geq 1$, then $k_\infty= \bigcup k_s$, where $k_s= k\QQ_{2, s}$,  is called the   cyclotomic $\ZZ_2$-extension of $k$.  
	The inverse limit $A(k_\infty):=\varprojlim A(k_s)$
	with respect to the norm maps is called the $2$-Iwasawa module of $k$, where $A(k)$   denotes the $2$-class group   of a number field $k$.
		For more details     we refer the reader to \cite[Chapter 13]{washington1997introduction}. 
	We would like to draw the reader's attention to the fact that the notation $k_s$ 
	 is not always used to denote the $s$-th layer of the cyclotomic $\ZZ_2$-extension of $k$; its meaning depends on the context.
	
	\bigskip

	For $n >1$,   it is known that $G$ is generated by two elements $a$ and $b$ and 
	admits a representation by generators and relations as follows:
	$$G=\langle  a, b\ | \ a^2\equiv b^{2^n} \equiv 1 \text{ mod } G'\rangle   .$$

	The $2$-group $G$ is classified (cf. \cite[Sec. 2 and 3]{benjashnepreprint1993}) in one of the following structures: {\bf abelian}, {\bf modular}, {\bf metacyclic-nonmodular} or   { \bf nonmetacyclic}. For convenience of readers, we recall that a modular group is defined as follows:
	\begin{maindefinition}
		A modular group $M$ is a $2$-group of order $2^m$, with $m >3$, that is generated by two elements $a$ and $b$ and admits the following presentation:
		$$M=\langle  a, b\ | \ a^2= b^{2^{m-1}} = 1 \text{ and  }  [a,b]= b^{2^{m-2}}\rangle  ,$$
		where $[a,b]$ is the commutator of $a$ and $b$.
	\end{maindefinition}
	
	\bigskip 
	Notice that,  {\bf for us, a  metacyclic-nonmodular is not abelian}. 
	The inverse Galois problem for this case was recently  investigated by some authors, including Azizi,   Rezzougui and       Zekhnini,  where they were interested in   the realizability  of finite metacyclic $2$-groups $G$ as the Galois group of the maximal unramified pro-$2$-extension of the cyclotomic $\mathbb Z_2$-extension of some number fields (cf. \cite{AziziRezzouguiZekhniniPeriodica,AziziRezzouguiZekhniniDebrecen}). However,  
	a finite  metacyclic $2$-groups $G$ such that $G^{ab}:=G/G'  \simeq\ZZ/2 \ZZ\times\ZZ/2^n \ZZ$ is in one of the following structures:  {\bf abelian}, {\bf modular} or {\bf metacyclic-nonmodular}. With this in mind, it is legitimate to ask the following question:
	
	\begin{center}
		{\bf Question 1:}	Are all these different structures of metacyclic $2$-groups realizable as the Galois group of the maximal unramified pro-$2$-extension of certain cyclotomic $\mathbb Z_2$-extensions ?
	\end{center}
	
	Moreover,    Benjamin and     Snyder \cite{benjashnepreprint1993} showed that   metacyclic-nonmodular groups whose abelianization is isomorphic to $\ZZ/2 \ZZ\times\ZZ/2^n \ZZ$ are classified as follows:
	\begin{mainproposition}[\cite{benjashnepreprint1993}, Proposition 2]
		Let $M$ be a  metacyclic-nonmodular group  whose abelianization is isomorphic to $\ZZ/2 \ZZ\times\ZZ/2^n \ZZ$, with $n >1$,  and such that   $M=\langle a, b\ | \ a^2\equiv b^{2^n} \equiv 1 \text{ mod } M' \rangle   $. Then $M$ admits a presentation in one of the following types, where $\alpha$, $s$ and $k$ are integers: 
		
		
		$\begin{array}{cl}
			Type  \ 1:&  a^{2^\alpha}=1,  \  b^{2^n}=1,\  b^{-1}ab=a^{-1} , \   \alpha >1.\\
			Type \ 2:&  a^{2^\alpha}=1, \ b^{2^n}=a^{2^{\alpha-1}} ,\ b^{-1}ab=a^{-1} ,  \ \alpha >1 .\\
			Type \ 3:& a^{2^\alpha}=1, \ b^{2^n}=1 ,\ b^{-1}ab=a^{-1+k2^s} ,  \ \alpha >1, \ \alpha >s >1, k \text{ is odd}  .\\
			Type \ 4:& a^{2^\alpha}=1, \ b^{2^n}=a^{2^{\alpha-1}} ,\ b^{-1}ab=a^{-1+k2^s} ,  \ \alpha >1, \ \alpha >s >1, k \text{ is odd}  . 
		\end{array} $ 
	\end{mainproposition}
	
	This also led us to ask the following second question:
	\begin{center}
		{\bf Question 2:}	Are all these four types  of metacyclic-nonmodular groups realizable as the Galois group of the maximal unramified pro-$2$-extension  of certain cyclotomic $\mathbb Z_2$-extensions (resp. $2$-extension of some number fields)?
	\end{center}
	
	In this paper, we investigate these two questions, which lead us to examine  Greenberg's conjecture for certain multiquadratic fields, focusing in particular on the $4$-rank of their $2$-Iwasawa module.
	In fact, the main results are summarized as follows:
	\bigskip

	The following result gives a necessary and sufficient condition for   families of real quadratic fields to have   
	a   {metacyclic-nonmodular pro-$2$-group of Type 1 with $\alpha=2$} as the Galois group of the maximal unramified pro-$2$-extension. Let $\eta$ be an element of $\{1,2\}$, $h_2(d)$ and
	$ \varepsilon_{d}$ be respectively the $2$-class number and the fundamental unit of a real quadratic number field $\QQ(\sqrt{d})$.  	Let $\delta_{a,b}$ be   the Kronecker delta, i.e.,  $\delta_{a,b}=1$ or $0$ according to whether $a=b$ or not.    We have:
	
	\begin{maintheorem}[Theorem \ref{etatheoremmetacyclic-nonmodular}]
		Let $ q\equiv     3\pmod 4$ and  $r\equiv s\equiv    5\pmod 8$   be  distinct   prime numbers such that $\left(\frac{q}{r}\right)=\left(\frac{q}{s}\right)= (-1)^{\delta_{\eta,2}}$, 
		$\left(\frac{r}{s}\right)=1$    and  $N(\varepsilon_{rs})=1$.  Put $F=\QQ(\sqrt{\eta qrs})$ and 	$ \varepsilon_{\eta qrs}=\gamma+\gamma'\sqrt{\eta qrs}$.
		Assume  that $2^{\delta_{\eta,1}}q(\gamma-1)$    is not a square  in $\NN$. 
		Then, we have: $$A(F)\simeq A(F_\infty)\simeq \ZZ/2\ZZ\times \ZZ/2^m\ZZ,$$
		where $m$  is such that $h_2(\eta qrs)=2^{m+1}$. Furthermore, the Galois group
		$\mathrm{Gal}(\mathcal{L}(F_\infty)/F_\infty)$,  where $\mathcal{L}(F_\infty)$ denotes the maximal unramified pro-$2$-extension  of $F_\infty$, is   
		a   {\bf metacyclic-nonmodular pro-$2$-group of Type 1 with $\alpha=2$} if and only if $\left(\frac{r}{s}\right)_4\not=\left(\frac{s}{r}\right)_4$.
	\end{maintheorem}

	The  stability of the $2$-rank of the class group of intermediate fields of the   cyclotomic $\ZZ_2$-extensions of quadratic and biquadratic fields 
	has been extensively investigated, especially in  \cite{ChemseddinGreenbergConjectureI,LaxmiSaikia}. However, the study of the stability of the $4$-rank of the class group in these extensions is very challenging.
	Our investigations expose techniques  that can be used to explore the stability of the $4$-rank of the class group in the   cyclotomic $\ZZ_2$-extensions.  
	In fact,   we have the following corollary that gives the construction of families of real  biquadratic fields satisfying Greenberg's conjecture and whose $2$-Iwasawa module is of $4$-rank $1$ and $2$.

	\begin{maincorollary}[Corollary \ref{corol2}]
		Let $ q\equiv     3\pmod 4$ and  $r\equiv s\equiv    5\pmod 8$   be  distinct   prime numbers such that $\left(\frac{q}{r}\right)=\left(\frac{q}{s}\right)= (-1)^{\delta_{\eta,2}}$,  
	$\left(\frac{r}{s}\right)=1$          and  $N(\varepsilon_{rs})=1$. Put $\varepsilon_{\eta qrs}=\gamma+\gamma'\sqrt{\eta qrs}$.	Assuming   that $2^{\delta_{\eta,1}}q(\gamma-1)$    is not a square  in $\NN$, we get the following items.
	\begin{enumerate}[$1)$]
		\item Let  $K'=\QQ(\sqrt{r},\sqrt{\eta qs})$, we have:
		$$ A(K'_\infty)\simeq A (K')\simeq \ZZ/2\ZZ\times \ZZ/2^m\ZZ.$$
		
		\item Let   $K =\QQ(\sqrt{\eta q},\sqrt{rs})$ and assume furthermore that $\left(\frac{r}{s}\right)_4\not=\left(\frac{s}{r}\right)_4$. We have:
		$$	A(K_\infty)\simeq A (K)\simeq {\ZZ/4\ZZ\times \ZZ/2^{m-1}\ZZ}.$$
	\end{enumerate}
	Here  $m$  is the integer such that $h_2(\eta qrs)=2^{m+1}$.    
	\end{maincorollary}
	
	\bigskip

	For  odd prime numbers $q$, $r$ and $s$,   the investigation of the $2$-Iwasawa module (resp. the $2$-class group) of the real triquadratic fields $\QQ(\sqrt{\eta q }, \sqrt {r},  \sqrt{s})$ represents a very significant challenge and all that is known about the $2$-Iwasawa module (resp. the $2$-class group) of these fields appears in  \cite[p. 1206, Corollary]{Ajn}, where the   author gave    some families of these fields, with $\eta=1$,  whose $2$-Iwasawa modules (resp. the $2$-class groups) are trivial, and in \cite[p. 2]{chems24}, where the   author constructed a family of these fields, with $\eta=1$, whose $2$-Iwasawa modules (resp. the $2$-class groups) are isomorphic to 
	$\ZZ/2\ZZ$. 
	
	Here  we construct a family of real    triquadratic fields $\mathbb{F}=	\QQ(\sqrt{\eta q},\sqrt{r},\sqrt{s})$, where $q$, $r$ and $s$ are all odd prime numbers,  
	that satisfy Greenberg's conjecture and both $A(\mathbb{F})$ and $ A(\mathbb{F}_\infty)$ are isomorphic to $\ZZ/2\ZZ\times \ZZ/2^{k}\ZZ$, where $k\geq 1$ is an integer.

	\begin{maintheorem}[Theorem \ref{corllarealtri(2 2)}]
			Let $ q\equiv     3\pmod 4$ and  $r\equiv s\equiv    5\pmod 8$   be  distinct   prime numbers such that $\left(\frac{q}{r}\right)=\left(\frac{q}{s}\right)= (-1)^{\delta_{\eta,2}}$, 
		$\left(\frac{r}{s}\right)=1$ and 	$\left(\frac{r}{s}\right)_4\not=\left(\frac{s}{r}\right)_4$. 
		Assume   that $2^{\delta_{\eta,1}}q(\gamma-1)$    is not a square  in $\NN$ and let $m$ be the integer such that $h_2(\eta qrs)=2^{m+1}$.
		Then, for $\mathbb{F}=\QQ(\sqrt{\eta q},\sqrt{r},\sqrt{s})$, we have:
		$$  A(\mathbb{F})\simeq A(\mathbb{F}_\infty)\simeq \ZZ/2\ZZ\times \ZZ/2^{m-1}\ZZ.$$
		Moreover, $\mathrm{Gal}(\mathcal{L}(\mathbb{F}_\infty)/\mathbb{F}_\infty)$, the Galois group of the maximal unramified pro-$2$-extension $\mathcal{L}(\mathbb{F}_\infty)$ over $\mathbb{F}_\infty$, is abelian.
	\end{maintheorem}
	
	\bigskip
	The plan of this paper is the following. In Section \ref{sec2}, we adopt some notations and collect some materials from groups and class field theories that we will use later.
	In Section \ref{sec3}, we present our main results and their proofs.

	\section{\bf Prerequisites from groups  and class field     theories}\label{sec2}	
	Before starting our investigations, we introduce additional notation   and to  collect some prerequisites from groups  and class field     theories.	Let $n\geq 2$ be a natural number and $G$ be a finite metabelian $2$-group of rank $2$ such that $G^{ab}:=G/G'  \simeq\ZZ/2 \ZZ\times\ZZ/2^n \ZZ$. Thus, it is known that $G$ is generated by two elements $a$ and $b$ and 
	admits a representation by generators and relations as follows:
	$$G=\langle  a, b\ | \ a^2\equiv b^{2^n} \equiv 1 \text{ mod } G'\rangle   .$$
	Moreover, $G$ admits three maximal subgroups of index $2$ denoted $H_{i2}$, for $i=1$, $2$ or $3$, and three normal subgroups of index $4$ denoted 
	$H_{i4}$, for $i=1$, $2$ or $3$, such  that $H_{34}=H_{12}\cap H_{22}\cap  H_{32}$, i.e. the intersection of all maximal subgroups of $G$. Notice that
	$H_{34}$ is the Frattini subgroup of $G$.
	Note that the groups  $H_{12}$, $H_{22}$ and
	$H_{3 2}$ are presented as follows:
	
	$$H_{12}=\langle b, G'\rangle , \quad  H_{22}=\langle ab, G'\rangle  \text{ and }
	H_{3 2}=\langle a, b^2, G'\rangle,$$\label{subgroups}
	and the groups   $H_{1 4}$,
	$H_{2 4}$ and $H_{34}$ are presented as follows:
	$$H_{1 4}=\langle a,b^4, G'\rangle, \quad H_{2 4}=\langle ab^2, G'\rangle \text{ and } H_{3 4}=\langle b^2, G'\rangle.$$

	Whenever $G$ a metacyclic-nonmodular $2$-groups, the authors of \cite{aaboune} gave explicitly the   subgroups of index $2$ or $4$, and their abelianizations as presented in Table \ref{tab1} and Table \ref{tab2}.	
	
	\begin{table}[H]
		\centering
		\renewcommand{\arraystretch}{1.2}
		\begin{tabular}{|c|c|c|c|c|c|}
			\hline
			$i$ & Structure of $G$ & $H_{i2}$ & $H'_{i2}$ & $H^{ab}_{i2}$ & $|G'|$ \\
			\hline
			1 & \multirow{3}{*}{Type 1 and $\alpha = 2$} 
			& $\langle a^2, b \rangle$ & $\langle 1 \rangle$ & $(2, 2^n)$ & \multirow{3}{*}{2} \\
			\cline{1-1} \cline{3-5}
			2 &  & $\langle a^2, ab \rangle$ & $\langle 1 \rangle$ & $(2, 2^n)$ & \\
			\cline{1-1} \cline{3-5} 
			3 &  & $\langle a, b^2 \rangle$ & $\langle 1 \rangle$ & $(4, 2^{n-1})$ & \\
			\hline 
			\hline
			1 & \multirow{2}{*}{Type 1, 2, 3 or 4 and $\alpha \geq 3$} 
			& $\langle a^2, b \rangle$ & $\langle a^4 \rangle$ & $(2, 2^n)$ & \multirow{2}{*}{$\geq 4$} \\
			\cline{1-1} \cline{3-5}
			2 &  & $\langle a^2, ab \rangle$ & $\langle a^4 \rangle$ & $(2, 2^n)$ &  \\
			\hline
			3 & \begin{tabular}[c]{@{}c@{}}Type 1, 3 and $s = \alpha - 1$, $\alpha \geq 3$\\ 
				Type 2, 4 and $s = \alpha -1$, $\alpha \geq 3$\\ 
				Type 3, 4 and $s < \alpha - 1$, $\alpha \geq 4$ \end{tabular}
			& $\langle a, b^2 \rangle$ 
			& \begin{tabular}[c]{@{}c@{}}$\langle 1 \rangle$\\$\langle 1 \rangle$\\  $\langle a^{2^{s+1}} \rangle$ \end{tabular} 
			& \begin{tabular}[c]{@{}c@{}}$(2^\alpha, 2^{n-1})$\\ $(2^{\alpha - \varepsilon}, 2^{n - \varepsilon'})$\\ $(2^{s+1}, 2^{n-1})$ \end{tabular}
			& $\geq 4$ \\
			\hline
		\end{tabular}
	\vspace{1em}
		\caption{Subgroups of index $2$ in $G$}\label{tab1}
	\end{table}
		Here \quad
	$
	\left\{
	\begin{array}{ll}
		\varepsilon = 0, \varepsilon' = 1 & \text{if } n \leq \alpha, \\
		\varepsilon = 1, \varepsilon' = 0 & \text{if } n > \alpha.
	\end{array}
	\right.$
	
	\begin{table}[H]\label{tab2}
		\centering
		\renewcommand{\arraystretch}{1.2}
		\setlength{\tabcolsep}{2.3pt}
		\begin{tabular}{|c|c|c|c|c|c|}
			\hline
			$i$ & Structure of $G$ & $H_{i4}$ & $H'_{i4}$ & $H^{ab}_{i4}$ & $|G'|$ \\
			\hline
			1 & \multirow{3}{*}{Type 1 and $\alpha = 2$} & $\langle a, b^4 \rangle$ & $\langle 1 \rangle$ & $(4, 2^{n-2})$ & \multirow{3}{*}{2} \\
			\cline{1-1} \cline{3-5}
			2 & & $\langle a^2, ab^2 \rangle$ & $\langle 1 \rangle$ & 
			$\begin{cases}
				(4), & \text{if } n=2 \\
				(2,2^{n-1}), & \text{if } n \geq 3
			\end{cases}$ & \\
			\cline{1-1} \cline{3-5}
			3 & & $\langle a^2, b^2 \rangle$ & $\langle 1 \rangle$ & $(2,2^{n-1})$ & \\
			\hline
		\end{tabular}
		
		\vspace{1em}

		\caption{Subgroups of index $4$ in $G$}
	\end{table}

	
	

	\bigskip
	
	Now let us convert the above group theory in terms of field extensions. Let $k$ be  a number field and consider
	$$k=k^{(0)} \subset k^{(1)} \subset  k^{(2)}  \subset \cdots\subset k^{(i)} \cdots \subset  \bigcup_{i\geq 0} k^{(i)}=\mathcal{L}(k)$$
	the      $2$-class field tower of $k$.   The field $\mathcal{L}(k)$ is called the maximal unramified pro-$2$-extension  of $k$ and for $k_\infty $, the cyclotomic $\ZZ_2$-extension of $k$, the group $G_{k_\infty}=\mathrm{Gal}(\mathcal{L}(k_\infty)/k_\infty)$ is isomorphic to the inverse limit $\varprojlim \mathrm{Gal}(\mathcal{L}(k_n)/k_n)$ with respect to the restriction map.
	\medskip

	Assume that $A(k)\simeq\ZZ/2 \ZZ\times\ZZ/2^n \ZZ$, where $n\geq2$ is a natural number.
	Put $G_k=\mathrm{Gal}(\mathcal{L}(k)/k)$. So by class field theory, we have 
	$G_k/G_k'\simeq\ZZ/2 \ZZ\times\ZZ/2^n \ZZ$ and  $G_k=\langle a,
	b\rangle$ such that $a^2\equiv b^{2^n}\equiv 1\mod G_k'$ and $A(k)=\langle \mathfrak{c}, \mathfrak{d}\rangle\simeq \langle aG_k',
	bG'\rangle$ where $(\mathfrak{c}, k^{(1)}/k)=aG_k'$
	and $(\mathfrak{d}, k^{(1)}/k)=bG_k'$ with $(\ .\ ,
	k^{(1)}/k)$ is the Artin symbol in the extension
	$k^{(1)}/k$.  	Moreover,   the group $G_k$ admits three normal subgroups of index $2$ (resp. $4$) denoted  $H_{i  2}$ (resp. $H_{i  4}$),  for $i\in \{1, 2,  3\}$  (cf. Page \pageref{subgroups}).
	
	Each subgroup $H_{ij}$ of $A(k)$
	corresponds to 	an unramified extension    ${F}_{i,j}$ of $k$ contained in
	$k^{(1)}$ such that $A(k)/H_{ij}\simeq
	\operatorname{Gal}(F_{i,j}/k)$ and
	$H_{ij}=\mathcal{N}_{{F}_{i,j}/k}(A(F_{i,j}) )$. Moreover, we can schematize the Hilbert $2$-class field tower of $k$ as in  Figure \ref{Fig2} (see below):
	\begin{figure}[H]
		$$
		\xymatrix{
			& k^{(2)} \ar@{<-}[d] & \\
			& k^{(1)}\ar@{<-}[ld]\ar@{<-}[d]\ar@{<-}[rd]\\
			{F_{1,  4}}\ar@{<-}[rd]&\ar@{<-}[ld] {F_{3,  4}}\ar@{<-}[d]\ar@{<-}[rd]  & {F_{2,  4}}\ar@{<-}[ld]\\
			{F_{1,  2}}\ar@{<-}[rd]& {F_{3,  2}} \ar@{<-}[d]  & {F_{2,  2}}\ar@{<-}[ld]\\
			&k
		}
		$$
		\caption{\label{Fig2}}
	\end{figure}
	
For additional details, we recommend that the reader consult \cite{aaboune,benjashnepreprint1993}.	
	Recall that a finite non-abelian group $M$ {\bf is said to be minimal} if all its proper subgroups are abelian. With the above hypothesis and notations we have:
	
	\begin{theorem}[\cite{aaboune}, Theorem 4.2]\label{AabounePrzekhiniNonmodular}
		Assume that the $2$-class group of ${F_{3,  2}}$ is isomorphic to  $\ZZ/2 \ZZ\times\ZZ/4 \ZZ$. Then the following assertions are equivalent :
		\begin{enumerate}[\rm $1)$]
			\item $G_k$ is minimal metacyclic-nonmodular $2$-group of type 1 and of order $16$,

			\item The $2$-class group of  ${F_{i, 2}}$ is isomorphic to  $\ZZ/2 \ZZ\times\ZZ/4 \ZZ$, with $i=1$ or $2$,
			\item The rank of the $2$-class group of ${F_{i,  2}}$ equals $2$ with $i=1$ or $2$.
		\end{enumerate}
	\end{theorem}

	\begin{lemma}[\cite{aaboune}, Corollary 4.9]\label{AabounePrzekhiniModularORabelian}
		Let $k$ be a number field whose $2$-class group is $A(k)\simeq\ZZ/2 \ZZ\times\ZZ/2^n \ZZ$, with $n\geq2$. Then $G_k$ is modular or abelian if and only if
		the  $\2r(A(F_{3,2}))=2$  and $\2r(A(F_{i,2}))=1$  for some $i\in\{1,2\}$.
	\end{lemma}

	\begin{lemma}[]\label{metacyclic-nonmodular}
		Let $k$ be a number field whose $2$-class group is $A(k)\simeq\ZZ/2 \ZZ\times\ZZ/2^n \ZZ$, with $n\geq2$. We have:
		\begin{enumerate}[\rm $1)$]
			\item $\2r(A(F_{i,2}))\leq 2$ for all $i\in\{1,2\}$.
			
			\item   $G_k$ is metacyclic-nonmodular if and only if
			$\2r(A(F_{i,2}))=2$    for all $i\in\{1,2,3\}$.
		\end{enumerate}
	\end{lemma}
	\begin{proof}
		This is a direct deduction of \cite[Theorem 3.1]{AziziRezzouguiZekhniniPeriodica}. 
	\end{proof}

	\bigskip
	
	Before closing this section, let us recall some further useful results.
	The   following class number formula for   multiquadratic number fields  which is usually attributed to Kuroda \cite{Ku-50} or Wada \cite{Wa-66}, but it goes back to Herglotz \cite{He-22}.
	\begin{lemma}[\cite{Ku-50}]\label{wada's f.}
		Let $k$ be a multiquadratic number field of degree $2^n$, where  $n\geq 2$ is an integer,  and $k_i$ the $s=2^n-1$ quadratic subfields of $k$. Then $h(k)$, the class number of $k$, is given by the following formula:
		\begin{eqnarray}\label{wadaf}
			 h(k)=\frac{1}{2^v}q(k)\prod_{i=1}^{s}h(k_i),
		\end{eqnarray} 
		where  $ q(k)=[E_k: \prod_{i=1}^{s}E_{k_i}]$ and   $$     v=\left\{ \begin{array}{cl}
			n(2^{n-1}-1); &\text{ if } k \text{ is real, }\\
			(n-1)(2^{n-2}-1)+2^{n-1}-1 & \text{ if } k \text{ is imaginary.}
		\end{array}\right.$$
	\end{lemma}

	\medskip

Let $k$ be 	a   multiquadratic field, and let     $\tau_1$ and 
	$\tau_2$  be two distinct elements of order $2$ of the Galois group of $k/\mathbb{Q}$. Set $\tau_3= \tau_1\tau_2$ and let $k_1$, $k_2$ and $k_3$ be the three subextensions of $k$ fixed by  $\tau_1$,
	$\tau_2$ and $\tau_3$, respectively.    \label{algo wada}
	Then the unit group of $k$  
	is generated by the elements of  $E_{k_1}$, $E_{k_2}$ and $E_{k_3}$, and the square roots of elements of   $E_{k_1}E_{k_2}E_{k_3}$ which are perfect squares in $k$. 
	This result, known as Wada's method \cite{Wa-66}, is useful for computing   the
	index $  q(k)$ appearing in \eqref{wadaf}.
	
	\bigskip

	We close this section with the following lemma which is a particular case of Fukuda's theorem  \cite{fukuda}.

	\begin{lemma}[\cite{fukuda}]\label{lm fukuda}
		Let $k_\infty/k$ be a $\mathbb{Z}_2$-extension and $n_0$  an integer such that any prime of $k_\infty$ which is ramified in $k_\infty/k$ is totally ramified in $k_\infty/k_{n_0}$.
		\begin{enumerate}[\rm $1)$]
			\item If there exists an integer $n\geq n_0$ such that   $h_2(k_n)=h_2(k_{n+1})$, then $h_2(k_n)=h_2(k_{m})$ for all $m\geq n$.
			\item If there exists an integer $n\geq n_0$ such that $\2r( A(k_n))= \2r(A(k_{n+1}))$, then
			$\2r(A(k_{m}))= \2r(A(k_{n}))$ for all $m\geq n$.
		\end{enumerate}
		Here $h_2(k)$ denotes the $2$-class number of a number field $k$.
	\end{lemma}

	\section{\bf The main results and their proofs}\label{sec3}
	
	  Let   $\eta$ and $\rho$  be two elements of  $\{1,2\}$ with $\eta\not=\rho$. Let  
	$ q\equiv     3\pmod 4$ and  $r\equiv s\equiv    5\pmod 8$   be  distinct prime numbers such that $\left(\frac{q}{r}\right)=\left(\frac{q}{s}\right)= (-1)^{\delta_{\eta,2}}$, where 
	$\delta_{a,b}$  denotes the Kronecker delta. 
	Denote by $\gamma$ and $\gamma'$   the integers such that  $\varepsilon_{\eta qrs}=\gamma+\gamma' \sqrt{\eta qrs}$.	  
		Put   
	$K=\QQ(\sqrt{\eta q},\sqrt{rs})$.  According to \cite[The main theorem, Items $5)$ and $14)$]{ChemseddinGreenbergConjectureI}, we have $\2r(A(K))=\2r(A(K_\infty))=2$.
	In this section, we aim to construct   a family of real biquadratic number fields of the form
	 $K$ such that $A(K)\simeq A(K_\infty)\simeq \ZZ/4\ZZ\times \ZZ/2^{m-1}\ZZ$, with $m\geq 2$. This leads to the construction of a family of real quadratic fields of the form  $F=\QQ(\sqrt{\eta qrs})$ such that     $A(F)\simeq A(F_\infty)\simeq \ZZ/2\ZZ\times \ZZ/2^{m}\ZZ$ and for which we give a necessary and sufficient condition for  the Galois group 
	$\mathrm{Gal}(\mathcal{L}(F_\infty)/F_\infty)$, where  $\mathcal{L}(F_\infty)$ denotes the  maximal  unramified pro-$2$-extension  of $F_\infty$, to be  a  
	metacyclic-nonmodular $2$-group of Type $1$ with $\alpha=2$.

	\begin{lemma}[\cite{Az-00}, Lemme 5]\label{lem2}
		Let $d>1$ be a square free integer and $\varepsilon_d=x+y\sqrt d$, where $x$, $y$ are  integers or semi-integers. If $N(\varepsilon_d)=1$, then $2(x+1)$, $2(x-1)$, $2d(x+1)$ and $2d(x-1)$ are not squares in $\QQ$.
	\end{lemma}

	\begin{lemma}\label{etaqrs}
		Let $ q\equiv     3\pmod 4$ and  $r\equiv s\equiv    5\pmod 8$   be  distinct prime numbers such that $\left(\frac{q}{r}\right)=\left(\frac{q}{s}\right)= (-1)^{\delta_{\eta,2}}$ and 
		$\left(\frac{r}{s}\right)=1$. Let   $\eta$ and $\rho$  be two elements of  $\{1,2\}$ with $\eta\not=\rho$.
		\begin{enumerate}[\rm $1)$]
			\item  Let $\gamma$   and $\gamma'$  be the integers such that
			$ \varepsilon_{\eta qrs}=\gamma+\gamma'\sqrt{\eta qrs}$.   Then  out of the three integers	 $2^{\delta_{\eta,1}}q(\gamma-1)$, $2r(\gamma+(-1)^{\delta_{\eta,2}})$ and  $2s(\gamma+(-1)^{\delta_{\eta,2}})$, exactly one of them is   a square in $\NN $. Moreover, we have:
			\begin{enumerate}[\rm $a)$]
				\item If $2^{\delta_{\eta,1}}q(\gamma-1)$ is a square in $\NN$, then $\sqrt{\eta\varepsilon_{\eta qrs}}= \gamma_1\sqrt{q} +\gamma_2\sqrt{\eta rs}  \text{ and } \eta=-q\gamma_1^2+\eta rs\gamma_2^2,$

				\item If $2r(\gamma+(-1)^{\delta_{\eta,2}})$ is a square in $\NN$, then $\sqrt{\eta\varepsilon_{\eta qrs}}= \gamma_1\sqrt{\eta r} +\gamma_2\sqrt{qs}$ and $\eta=(-1)^{\delta_{\eta,2}} \eta r\gamma_1^2+(-1)^{\delta_{\eta,1}}qs\gamma_2^2,$

				\item If $2s(\gamma+(-1)^{\delta_{\eta,2}})$  is a square in $\NN$, then $\sqrt{\eta\varepsilon_{\eta qrs}}= \gamma_1\sqrt{\eta s} +\gamma_2\sqrt{qr}$ and $\eta=(-1)^{\delta_{\eta,2}} \eta s\gamma_1^2+(-1)^{\delta_{\eta,1}}qr\gamma_2^2$.
				
			\end{enumerate}
			Here $\gamma_1$ and $\gamma_2$ are two integers.
			
			\item Let $x$ and $y$   be the integers such that
			$\varepsilon_{\rho qrs}=x+y\sqrt{\rho qrs}$. Then  out of the two integers    $2^{{\delta_{\rho,2}}}q(x+1)$ and  $2^{{\delta_{\rho,2}}}q(x-1)$, exactly one of them is   a square in $\NN $. Moreover, we have:
			\begin{enumerate}[\rm $a)$]

				\item If     $2^{{\delta_{\rho,2}}}q(x+1)$ is a square in $\NN$, then  $\sqrt{2\varepsilon_{\rho qrs}}= y_1\sqrt{\rho q} +y_2\sqrt{rs}  \text{ and } 2=\rho qy_1^2-rsy_2^2.$
				
				\item 	If  $2^{{\delta_{\rho,2}}}q(x-1)$ is a square in $\NN$, then  $\sqrt{2\varepsilon_{\rho qrs}}= y_1\sqrt{\rho q} +y_2\sqrt{rs}  \text{ and } 2=-\rho qy_1^2+rsy_2^2.$ 
				
			\end{enumerate}
			Here $y_1$ and $y_2$ are two integers.

		\end{enumerate}
	\end{lemma}	 
	\begin{proof}Let us first assume that $\eta=1$ $($i.e. $\rho=2)$. 
		\begin{enumerate}[\rm1)]
			\item 
			As $ N(\varepsilon_{qrs})=1 $,  by the unique factorization  of $ \gamma^{2}-1=qrs\gamma'^{2} $ in $ \mathbb{Z} $, and Lemma \ref{lem2}, there exist $\gamma_1$ and $\gamma_2$ in $\mathbb{Z}$ such that  we have one of the following systems:
			$$(1):\ \left\{ \begin{array}{ll}
				\gamma\pm1=2q\gamma_1^2\\
				\gamma\mp1=2rs\gamma_2^2,
			\end{array}\right.  \quad
			(2):\ \left\{ \begin{array}{ll}
				\gamma\pm1=2r\gamma_1^2\\
				\gamma\mp1=2qs\gamma_2^2,
			\end{array}\right. \quad
			(3):\ \left\{ \begin{array}{ll}
				\gamma\pm1=2s\gamma_1^2\\
				\gamma\mp1=2qr\gamma_2^2,
			\end{array}\right. 
			$$

			$$   
			(4):\ \left\{ \begin{array}{ll}
				\gamma\pm1=q\gamma_1^2\\
				\gamma\mp1=rs\gamma_2^2,
			\end{array}\right. \quad (5):\ \left\{ \begin{array}{ll}
				\gamma\pm1=r\gamma_1^2\\
				\gamma\mp1=qs\gamma_2^2,
			\end{array}\right.\text{   }(6):\ \left\{ \begin{array}{ll}
				\gamma\pm1=s\gamma_1^2\\
				\gamma\mp1=qr\gamma_2^2
			\end{array}\right.
			$$
			
		$$\text{ or }(7):\ \left\{ \begin{array}{ll}
			\gamma\pm1= \gamma_1^2\\
			\gamma\mp1=qrs\gamma_2^2.
		\end{array}\right.
		$$

			\begin{enumerate}[\rm$\bullet$]
				\item  Assume that the system $\left\{ \begin{array}{ll}
					\gamma+1=2q\gamma_1^2\\
					\gamma-1=2rs\gamma_2^2,
				\end{array}\right.$ holds. We have:	
				\[\left(\dfrac{2}{q}\right)=\left(\dfrac{2rs\gamma_2^2}{q}\right)=\left(\dfrac{\gamma-1}{q}\right)=\left(\dfrac{\gamma+1-2}{q}\right)=\left(\dfrac{2q\gamma_1 ^2-2}{q}\right)=\left(\dfrac{-2}{q}\right)=-\left(\dfrac{2}{q}\right),
				\]
				which is absurd. So this system can not hold.
				
				\item Assume that the system $\left\{ \begin{array}{ll}
					\gamma\pm1=q\gamma_1^2\\
					\gamma\mp1=rs\gamma_2^2 
				\end{array}\right.$ holds. We have:	
				\[1=\left(\dfrac{q\gamma_1^2}{s}\right)=\left(\dfrac{\gamma\pm1}{s}\right)=\left(\dfrac{\gamma\mp1\pm2}{s}\right) =\left(\dfrac{2}{s}\right)=-1,
				\]
				which is absurd. So this system also can not hold.
			\end{enumerate}
		We eliminate the remaining systems in the same manner, except the following:
			$$(1):\ \left\{ \begin{array}{ll}
				\gamma-1=2q\gamma_1^2\\
				\gamma+1=2rs\gamma_2^2,
			\end{array}\right.  \quad \left\{ \begin{array}{ll}
				\gamma+1=2r\gamma_1^2\\
				\gamma-1=2qs\gamma_2^2
			\end{array}\right. \quad \text{ and } \left\{ \begin{array}{ll}
				\gamma+1=2s\gamma_1^2\\
				\gamma-1=2qr\gamma_2^2.
			\end{array}\right.$$
			By subtracting and summing the equations of each system, we get the result of our first item for $\eta=1$. 
			\item For the second item, we proceed similarly and eliminate all systems except $\left\{ \begin{array}{ll}
				x\pm1=2qy_1^2\\
				x\mp1=rsy_2^2,
			\end{array}\right. $ with $y_1$ and $y_2$ are two integers such that $y= y_1y_2$. This  gives the second item for $\eta=1$ $($i.e. $\rho=2)$.
		\end{enumerate}
		The proof is analogous for the case $\eta=2$.
	\end{proof}

	\begin{example}  \label{exam3.3}
		Using 
		PARI/GP (cf. \cite{PARI}), we  get the following examples   illustrating the hypothesis of  Lemma \ref{etaqrs}.
	In the below example let  $\gamma$, $\gamma'$, $x$ and $y$  be the integers such that
		$ \varepsilon_{\eta qrs}=\gamma+\gamma'\sqrt{\eta qrs}$ and $ \varepsilon_{\rho qrs}=x+y\sqrt{\rho qrs}$. We have:
		\begin{enumerate}[$1)$]
			\item \begin{enumerate}[$a)$]
				\item	For $\eta=1$,  $q=23$, $r=13$ and $s=29$, we have $\left(\frac{23}{13}\right)=\left(\frac{23}{29}\right)= 1$,  
			$\left(\frac{13}{29}\right)=1$,  
			   $\gamma=2403974952$,  $\gamma'=25816380$, $x=40559167$ and $y=307992$. Moreover,  $  2  s(\gamma+1)$  and  $  2  q(x-1)$ are  squares in $\NN$. In fact, we have:
			    $$  2  s(\gamma+1)   =2\cdot 29\cdot(2^3\cdot3^2\cdot29^3\cdot37^2) \text{ and }2  q(x-1)=2\cdot 23\cdot (2\cdot 3^2\cdot23\cdot 313^2).$$

				\item	For $\eta=2$,  $q=31$, $r=53$ and $s=13$, we have  $\left(\frac{31}{53}\right)=\left(\frac{31}{13}\right)=  1$, 
			$\left(\frac{53}{13}\right)=1$,
			$\gamma=12401$,  $\gamma'=60$, $x=33760$ and $y=231$. Moreover,  $q(\gamma-1)$  and  $   q(x-1)$ are  squares in $\NN$.  
			    In fact, we have:
			     $$  q(\gamma-1)   =31\cdot(2^4\cdot5^2\cdot31) \text{ and }  q(x-1)= 31\cdot (  3^2\cdot11^2\cdot 31).$$

			    \end{enumerate}

			\item In the following tables (cf. Tables \ref{tableexamplesqs00} and \ref{tableexamplesqs002}) we give numerical examples satisfying the conditions of each possibility in   Lemma \ref{etaqrs}.
			 
		 	\begin{table}[H]
				{ \footnotesize
					$$ \begin{tabular}{  |c |c |c| c |c|  }
						\hline	\rsp  $(\eta, q,r,s)$  & $(1, 43,53,13)$ &$(1, 131,53,13)$ & $(1,211,53,13)$    \\ 
						\hline

						\rsp	\text{The case}    &   $2q(\gamma-1) \in \NN^2$  & $2r(\gamma+1) \in \NN^2$ & $2s(\gamma+1) \in \NN^2$      \\
						\hline 
						 \hline 
						  	\rsp  $(\eta, q,r,s)$ & $(2, 31,53,13)$ &$(2, 67,53,13)$ & $(2, 19,53,13)$    \\ 
						 \hline

						 \rsp	\text{The case}    &   $q(\gamma-1) \in \NN^2$  & $2r(\gamma-1) \in \NN^2$ & $2s(\gamma-1) \in \NN^2$       \\
						  \hline
						
					\end{tabular}$$}
				\caption{Examples on  $2^{\delta_{\eta,1}}q(\gamma-1)$, $2r(\gamma+(-1)^{\delta_{\eta,2}})$ or  $2s(\gamma+(-1)^{\delta_{\eta,2}}) \in \NN^2$} \label{tableexamplesqs00}
				
			\end{table}

			\begin{table}[H]
			{ \footnotesize
				$$ \begin{tabular}{  |c |c |c| c |c|  }
					\hline	\rsp  $(\eta, q,r,s)$  & $(1, 43,53,13)$ &$(1, 199,53,13)$      \\ 
					\hline

					\rsp	\text{The case}    &   $  2q(x+1) \in \NN^2$  & $2q(x-1) \in \NN^2$       \\
					\hline 
					\hline 
					\rsp  $(\eta, q,r,s)$ & $(2, 19,53,13)$ &$(2, 71,53,13)$     \\ 
					\hline

					\rsp	\text{The case}    &   $q(x+1) \in \NN^2$  & $q(x-1) \in \NN^2$        \\
					\hline
					
				\end{tabular}$$}
			\caption{Examples on  $2^{{\delta_{\rho,2}}}q(x+1)$ or  $2^{{\delta_{\rho,2}}}q(x-1)\in \NN^2$} \label{tableexamplesqs002}
			
		\end{table}	
	 	\end{enumerate}
		
	\end{example}

	\begin{corollary}\label{unitqrs}
		Let $ q\equiv     3\pmod 4$ and  $r\equiv s\equiv    5\pmod 8$   be  distinct   prime numbers such that $\left(\frac{q}{r}\right)=\left(\frac{q}{s}\right)= (-1)^{\delta_{\eta,2}}$ and 
		$\left(\frac{r}{s}\right)=1$. Put $F=\QQ(\sqrt{\eta qrs})$ and $ \varepsilon_{\eta qrs}=\gamma+\gamma'\sqrt{\eta qrs}$. Then $q(F_1)=1$ or $2$. Furthermore, we have:
		$$q(F_1) =1\text{ if and only if }  2^{\delta_{\eta,1}}q(\gamma-1)\not\in \NN^2.$$
	\end{corollary}	
	\begin{proof}
		This follows from  Lemma \ref{etaqrs} and Wada's method   (cf. Page \pageref{wada's f.}) applied to   $F_1=\QQ(\sqrt{qrs},\sqrt{2})$. In fact, we deduce that
		a fundamental system of units of $F_1$ is $\{\varepsilon_{2},\varepsilon_{\eta qrs}, \sqrt{\varepsilon_{\eta qrs}\varepsilon_{\rho qrs}}\}$ or $\{\varepsilon_{2},\varepsilon_{\eta qrs},  \varepsilon_{\rho qrs} \}$ according to whether $2^{\delta_{\eta,1}}q(\gamma-1)\in \NN^2$ or not.
	\end{proof}
	\begin{corollary}\label{corllaryh(F)=h(qrs)}
		Let $ q\equiv     3\pmod 4$ and  $r\equiv s\equiv    5\pmod 8$   be  distinct   prime numbers such that $\left(\frac{q}{r}\right)=\left(\frac{q}{s}\right)= (-1)^{\delta_{\eta,2}}$ and 
		$\left(\frac{r}{s}\right)=1$. Put $F=\QQ(\sqrt{\eta qrs})$ and $ \varepsilon_{\eta qrs}=\gamma+\gamma'\sqrt{\eta qrs}$. Then $h_2(F_1)=h_2(\eta qrs)$ or $2h_2(\eta qrs)$ and 
		$$h_2(F_1) =h_2(\eta qrs)\text{ if and only if }  2^{\delta_{\eta,1}}q(\gamma-1)\not\in \NN^2.$$
	\end{corollary}	
	\begin{proof}
		As $h_2(2)=1$, the class number formula gives 	$h_2(F_1)=\frac{q(F_1)}{4} h_2(qrs)h_2({2qrs}).$
		We have    
		$h_2(qrs)h_2({2qrs})=4h_2({ \eta qrs})$ 
		(cf. \cite[Proposition 1 and 2]{BenjShn(22)}). Therefore,
		$h_2(F_1)=h_2(\eta qrs)$ if and only if $q(F_1)=1$. So the result by the previous corollary. 
	\end{proof}	
	
	\begin{remark}
		Let $ q\equiv     3\pmod 4$ and  $r\equiv s\equiv    5\pmod 8$   be  distinct   prime numbers such that $\left(\frac{q}{r}\right)=\left(\frac{q}{s}\right)= (-1)^{\delta_{\eta,2}}$ and 
		$\left(\frac{r}{s}\right)=1$. By Lemma \ref{etaqrs}, the condition $2^{\delta_{\eta,1}}q(\gamma-1)\not\in \NN^2$
		is equivalent to $2r(\gamma+(-1)^{\delta_{\eta,2}})$ or  $2s(\gamma+(-1)^{\delta_{\eta,2}})$    is  a square in $\NN$. 
	\end{remark}

	\begin{remark}\label{calcul} Keep the same hypothesis of Corollary \ref{corllaryh(F)=h(qrs)}.    According to \cite[Remark 2.3]{AzZektaous}, $\sqrt{\varepsilon_{2}\varepsilon_{rs}\varepsilon_{2rs}}$ takes one of the following forms :	
		\[
		\left \{
		\begin{array}{ccc}
			\sqrt{\varepsilon_{2}\varepsilon_{rs}\varepsilon_{2rs}} &=& A+B\sqrt{2}+C\sqrt{rs}+D\sqrt{2rs}, \qquad  \qquad \quad \quad (F1) \\
			& &\\
			\sqrt{\varepsilon_{2}\varepsilon_{rs}\varepsilon_{2rs}} &=& E\sqrt{r}+F\sqrt{s}+G\sqrt{2r}+H\sqrt{2s}, \qquad  \qquad \quad (F2) 
		\end{array}
		\right.
		\]
		
		where $A,B,...,H \in \QQ.$
	\end{remark} 	
	
	\begin{lemma}\label{lemmunitstriquad}
		Let $ q\equiv     3\pmod 4$ and  $r\equiv s\equiv    5\pmod 8$   be  distinct   prime numbers such that $\left(\frac{q}{r}\right)=\left(\frac{q}{s}\right)= (-1)^{\delta_{\eta,2}}$, 
		$\left(\frac{r}{s}\right)=1$. 	Put   $K_1=\QQ(\sqrt{2},\sqrt{q}, \sqrt{rs})$ and  $ \varepsilon_{\eta qrs}=\gamma+\gamma'\sqrt{\eta qrs}$. Assume that $2^{\delta_{\eta,1}}q(\gamma-1)$ is not a square in  $\NN$.
		We have:
		\begin{enumerate}[\rm1)]
			\item   If $N(\varepsilon_{rs})=1$, then	$$E_{K_1}=\langle -1, \varepsilon_2,\varepsilon_{2rs}, \varepsilon_{\eta qrs}, \sqrt{\varepsilon_q},\sqrt{\varepsilon_{2q}}, \sqrt{\varepsilon_{\rho qrs}},\sqrt{\varepsilon_{rs}\varepsilon_{\eta qrs}} \rangle.$$
			
			\item	If  $N(\varepsilon_{rs})=-1$  and $\sqrt{\varepsilon_2\varepsilon_{rs}\varepsilon_{2rs}}\in \mathbb{Q}(\sqrt{2},  \sqrt{rs})$,
			then  	$$E_{K_1}=\langle -1, \varepsilon_2,\varepsilon_{rs}, \varepsilon_{\eta qrs}, \sqrt{\varepsilon_q},\sqrt{\varepsilon_{2q}}, \sqrt{\varepsilon_{\rho qrs}},\sqrt{\varepsilon_{2}\varepsilon_{rs}\varepsilon_{2rs}} \rangle.$$
			
			\item	If  $N(\varepsilon_{rs})=-1$  and $\sqrt{\varepsilon_2\varepsilon_{rs}\varepsilon_{2rs}}\notin \mathbb{Q}(\sqrt{2},  \sqrt{rs})$,
			then 	$$E_{K_1}=\langle -1, \varepsilon_2,\varepsilon_{rs}, \varepsilon_{\eta qrs}, \sqrt{\varepsilon_q},\sqrt{\varepsilon_{2q}}, \sqrt{\varepsilon_{\rho qrs}},\sqrt{\varepsilon_{2}\varepsilon_{rs}\varepsilon_{2rs}\varepsilon_{\eta qrs}} \rangle.$$  
		\end{enumerate}
		Thus, $q(K_1)=2^4$.
	\end{lemma}
	\begin{proof} 	Let $ q\equiv     3\pmod 4$ and  $r\equiv s\equiv    5\pmod 8$   be   prime numbers such that $\left(\frac{q}{r}\right)=\left(\frac{q}{s}\right)= (-1)^{\delta_{\eta,2}}$ and 
		$\left(\frac{r}{s}\right)=1$. We shall use the method of Wada  presented in Page  \pageref{algo wada}. Consider the three  biquadratic subfields of $K_1$ defined by
		$k_1=\mathbb{Q}(\sqrt{2},  \sqrt{q})$, $k_2=\mathbb{Q}(\sqrt{2},  \sqrt{rs})$ and $k_3=\mathbb{Q}(\sqrt{2},  \sqrt{qrs})$. 
		Let $\tau_1$, $\tau_2$ and $\tau_3$ be the elements of  $ \mathrm{Gal}(K_1/\QQ)$ defined by
		\begin{center}	\begin{tabular}{l l l }
				$\tau_1(\sqrt{2})=-\sqrt{2}$, \qquad & $\tau_1(\sqrt{q})=\sqrt{q}$, \qquad & $\tau_1(\sqrt{rs})=\sqrt{rs},$\\
				$\tau_2(\sqrt{2})=\sqrt{2}$, \qquad & $\tau_2(\sqrt{q})=-\sqrt{q}$, \qquad &  $\tau_2(\sqrt{rs})=\sqrt{rs},$\\
				$\tau_3(\sqrt{2})=\sqrt{2}$, \qquad &$\tau_3(\sqrt{q})=\sqrt{q}$, \qquad & $\tau_3(\sqrt{rs})=-\sqrt{rs}.$
			\end{tabular}
		\end{center} 
		Notice that  $\mathrm{Gal}(K_1/\QQ)=\langle \tau_1, \tau_2, \tau_3\rangle$
		and that the subfields  $k_1$, $k_2$ and $k_3$ are
		fixed by  $\langle \tau_3\rangle$, $\langle\tau_2\rangle$ and $\langle\tau_2\tau_3\rangle$ respectively. Therefore according to Wada's algorithm a fundamental system of units  of $K_1$ consists  of seven  units chosen from those of $k_1$, $k_2$ and $k_3$, and  from the square roots of the elements of $E_{k_1}E_{k_2}E_{k_3}$ which are squares in $K_1$.	  
		We have:  
		\begin{enumerate}[$\bullet$]
			\item A fundamental system of units  of 
			$k_1$ is $\{\varepsilon_2,\sqrt{\varepsilon_{q}}, \sqrt{\varepsilon_{2q}}\}$. In fact, we apply Wada’s algorithm and use the representations $ \sqrt{2\varepsilon_{q}}=c+d\sqrt{q}$ and $ \sqrt{2\varepsilon_{2q}}=u+v\sqrt{2q}$, where $c,d,u$ and $v$ are integers such that 
			$$2(-1)^\alpha=c^2- qd^2  \text{ and } 2(-1)^\alpha=u^2-2qv^2,$$  with $\alpha =1$ or $0$ according to whether $q\equiv 3$ or $7\pmod 8$  (cf.   \cite[Lemma 2.5]{chemseddineArith} and \cite[Lemma 3.2]{chemsmohahajjami}).
			
			\item If $N(\varepsilon_{rs})=1$, then a fundamental system of units of $k_2$ is $\{\varepsilon_2,\varepsilon_{rs}, \varepsilon_{2rs}\}$ (cf. \cite[Theorem 2.2.]{AzZektaous}).
			\item If $N(\varepsilon_{rs})=-1$, then a fundamental system of units of $k_2$ is $\{\varepsilon_2,\varepsilon_{rs}, \sqrt{\varepsilon_2\varepsilon_{rs}\varepsilon_{2rs}}\}$ or $\{\varepsilon_2,\varepsilon_{rs}, 
			\varepsilon_{2rs}\}$, according to ${\varepsilon_2\varepsilon_{rs}\varepsilon_{2rs}}$ is a square in $k_2$ or not (cf. \cite[Theorem 2.2.]{AzZektaous}).
			\item By Corollary \ref{unitqrs} a fundamental system of units of $k_3$ is  $\{\varepsilon_2,\varepsilon_{qrs}, \sqrt{\varepsilon_{qrs}\varepsilon_{2qrs}}\}$ or $\{\varepsilon_2,\varepsilon_{qrs}, \varepsilon_{2qrs}\}$ according to $2^{\delta_{\eta,1}}q(\gamma-1)$ is a square in $\NN$ or not. Here $\gamma$   and $\gamma'$  are the integers such that $ \varepsilon_{\eta qrs}=\gamma+\gamma'\sqrt{\eta qrs}$.
		\end{enumerate}
		Assume that $2^{\delta_{\eta,1}}q(\gamma-1)$ is not a square in $\NN$. So we have three cases to distinguish.
		\begin{enumerate}[$1)$]
			\item Assume that $N(\varepsilon_{rs})=1$. So 
			\begin{equation}
				E_{k_1}E_{k_2}E_{k_3}=\langle-1,  \varepsilon_{2}, \varepsilon_{rs},\varepsilon_{2rs}, \varepsilon_{qrs},\varepsilon_{2qrs} ,\sqrt{\varepsilon_{q}},\sqrt{\varepsilon_{2q}}\rangle.
			\end{equation}
			Notice that according to Lemma \ref{etaqrs} we have $\sqrt{\varepsilon_{\rho qrs}}\in K_1$. 
			Let us find the elements $\chi$   of $ K_1$ which are the  square root of an element of $E_{k_1}E_{k_2}E_{k_3}$. 
			Therefore, we can assume that		
			\begin{equation}\label{eq1}
				\chi^2=\varepsilon_{2}^a\varepsilon_{rs}^b \varepsilon_{2rs}^c\varepsilon_{\eta qrs}^d\sqrt{\varepsilon_{q}}^e\sqrt{\varepsilon_{2q}}^f,
			\end{equation}
			
			where $a, b, c, d, e$ and $f$ are in $\{0, 1\}$. We shall use the norm maps from $K$ to its subfields to eliminate  the equations  which do not occur.  
  	We note that the unit 	 $\varepsilon_{2rs}$ is known to have negative norm $($see for example   \cite[Corollary $3.6$,  p. 98]{AziTaouss}$)$. We have:
			
			\noindent\ding{229} By   applying the norm $N_{K_1/\QQ(\sqrt{q},\sqrt{rs})}=1+\tau_1$, we get:
			\begin{eqnarray*}
				N_{K_1/\QQ(\sqrt{q},\sqrt{rs})}(\chi^2)&=& (-1)^a \cdot\varepsilon_{rs}^{2b}\cdot (-1)^c\cdot ( \varepsilon_{\eta qrs})^{2d\delta_{\eta,1}}\cdot \varepsilon_{q}^e\cdot (-1)^{e}\cdot (-1)^{(\alpha+1)f }. 	\end{eqnarray*}
			Thus $a+c+ e +(\alpha+1)f\equiv 0\pmod2$. As $\varepsilon_q$ is not a square in $\QQ(\sqrt{q},\sqrt{rs})$, we get $e=0$. Then  $a+c+ (\alpha+1)f\equiv 0\pmod2$.
			Therefore, we have:
			\begin{equation*} 
				\chi^2=\varepsilon_{2}^a\varepsilon_{rs}^b \varepsilon_{2rs}^c\varepsilon_{\eta qrs}^d \sqrt{\varepsilon_{2q}}^f.
			\end{equation*}
			
			\noindent\ding{229} By   applying the norm $N_{K_1/\QQ(\sqrt{2q},\sqrt{rs})}=1+\tau_1\tau_2$, we get:
			\begin{eqnarray*}
				N_{K_1/\QQ(\sqrt{2q},\sqrt{rs})}(\chi^2)&=& (-1)^a \cdot\varepsilon_{rs}^{2b}\cdot (-1)^c\cdot ( \varepsilon_{\eta qrs})^{2d\delta_{\eta,2}}\cdot \varepsilon_{2q}^f\cdot (-1)^{f}.
			\end{eqnarray*}
			Then $a+c+ f\equiv 0\pmod2$. As $\varepsilon_{2q}$ is not a square in $\QQ(\sqrt{2q},\sqrt{rs})$, we get $f=0$ and so  $a=c.$ Thus, we have:
			\begin{equation*}
				\chi^2=\varepsilon_{2}^a\varepsilon_{rs}^b \varepsilon_{2rs}^a\varepsilon_{\eta qrs}^d.
			\end{equation*}
			\noindent\ding{229} By   applying the norm $N_{K_1/\QQ(\sqrt{2},\sqrt{q})}=1+\tau_3$, we get:
			\begin{eqnarray*}
				N_{K_1/\QQ(\sqrt{2},\sqrt{q})}(\chi^2)&=& \varepsilon_2^{2a} \cdot (-1)^a\cdot 1.
			\end{eqnarray*}
			Then $a=0$. Therefore $\chi^2= \varepsilon_{rs}^b \varepsilon_{\eta qrs}^d$. Notice that  $\sqrt{\varepsilon_{rs}}  \notin K_1$, in fact we deduce from \cite[Proof of Proposition 2.5]{AzZektaous} that
			 	$\sqrt{\varepsilon_{rs}}=\alpha_1\sqrt{r}+\alpha_2\sqrt{s}	$ for some integers $\alpha_1$ and $\alpha_2$. Moreover, we have
			$  \sqrt{\varepsilon_{\eta qrs}} \notin K_1$  and  $\sqrt{\varepsilon_{rs}}\sqrt{\varepsilon_{\eta qrs}}\in K_1$  (cf. Lemma \ref{etaqrs}). Thus,   the only squarefree elements  in $E_{k_1}E_{k_2}E_{k_3}$  which are squares in $K_1$ are ${\varepsilon_{rs}}{\varepsilon_{\eta qrs}}$ and ${\varepsilon_{\rho qrs}}$. Hence, 
		 	$$E_{K_1}=\langle -1, \varepsilon_2,\varepsilon_{2rs}, \varepsilon_{\eta qrs}, \sqrt{\varepsilon_q},\sqrt{\varepsilon_{2q}}, \sqrt{\varepsilon_{\rho qrs}},\sqrt{\varepsilon_{rs}\varepsilon_{\eta qrs}} \rangle.$$

			\item Assume that  $N(\varepsilon_{rs})=-1$ and $\sqrt{\varepsilon_2\varepsilon_{rs}\varepsilon_{2rs}}\in k_2$.
			Then 
			\begin{eqnarray} E_{k_1}E_{k_2}E_{k_3}=\langle-1,  \varepsilon_{2}, \varepsilon_{rs},\varepsilon_{qrs},\varepsilon_{2qrs},\sqrt{\varepsilon_{q}},\sqrt{\varepsilon_{2q}}, \sqrt{\varepsilon_{2}\varepsilon_{rs}\varepsilon_{2rs}  }\rangle.\end{eqnarray}
			
			As above, we   consider $\chi$ an element of $K_1$ of the form:
			$$\chi^2=\varepsilon_{2}^a\varepsilon_{rs}^b \varepsilon_{\eta qrs}^c\sqrt{\varepsilon_{q}}^d\sqrt{\varepsilon_{2q}}^e
			\sqrt{\varepsilon_{2}\varepsilon_{rs}\varepsilon_{2rs}}^f,$$
			where $a, b, c, d, e$ and $f$ are in $\{0, 1\}$.

			\noindent\ding{229} Let us start by applying   the norm $N_{K_1/\QQ(\sqrt{q},\sqrt{rs})}=1+\tau_1$. Notice that
			$N_{K_1/\QQ(\sqrt{q},\sqrt{rs})}(\sqrt{\varepsilon_{2}\varepsilon_{rs}\varepsilon_{2rs}})=(-1)^{a_1}  \varepsilon_{rs}$, for some $a_1\in \{0,1\}$.
			 We have:
			\begin{eqnarray*}
				N_{K_1/\QQ(\sqrt{q},\sqrt{rs})}(\chi^2)&=& (-1)^a \cdot\varepsilon_{rs}^{2b}\cdot (\varepsilon_{\eta qrs})^{2c\delta_{\eta,1}}\cdot (-1)^{d}\varepsilon_{q}^{d}\cdot (-1)^{(\alpha+1)e }\cdot \varepsilon_{rs}^f\cdot (-1)^{a_1 f}. 	\end{eqnarray*}
			Thus $a+d +(\alpha+1)e+a_1 f\equiv 0\pmod2$. As $\varepsilon_q$  and $\varepsilon_{rs}$ are not  squares in $\QQ(\sqrt{q},\sqrt{rs}),$ we have $d=f$. Therefore,
			$$\chi^2=\varepsilon_{2}^a\varepsilon_{rs}^b \varepsilon_{\eta qrs}^c\sqrt{\varepsilon_{q}}^d\sqrt{\varepsilon_{2q}}^e
			\sqrt{\varepsilon_{2}\varepsilon_{rs}\varepsilon_{2rs}}^d.$$
			
			\noindent\ding{229} Now let us apply the norm $N_{K_1/\QQ(\sqrt{2},\sqrt{q})}=1+\tau_3$.
			 As
			$N_{K_1/\QQ(\sqrt{2},\sqrt{q})}(\sqrt{\varepsilon_{2}\varepsilon_{rs}\varepsilon_{2rs}})=(-1)^{a_2}  \varepsilon_{2}$ with $a_2\in \{0,1\}$, 
			 we have:
			\begin{eqnarray*}
				N_{K_1/\QQ(\sqrt{2},\sqrt{q})}(\chi^2)&=& \varepsilon_2^{2a} \cdot (-1)^b \cdot \varepsilon_{q}^{d}\cdot \varepsilon_{2q}^{e}\cdot \varepsilon_{2}^{d}\cdot (-1)^{a_2 d}.
			\end{eqnarray*} 
			Thus, $b +a_2 d\equiv 0\pmod2$.          
			As $\varepsilon_2$ is not a square in $\QQ(\sqrt{2},\sqrt{q})$, we get $d=0$ and so $b=0$. Therefore,
			$$\chi^2=\varepsilon_{2}^a  \varepsilon_{\eta qrs}^c \sqrt{\varepsilon_{2q}}^e  .$$
			
			\noindent\ding{229} Finally by applying the norm $N_{K_1/\QQ(\sqrt{2q},\sqrt{rs})}=1+\tau_1\tau_2$, we get:
			\begin{eqnarray*}
				N_{K_1/\QQ(\sqrt{2q},\sqrt{rs})}(\chi^2)&=& (-1)^{a} \cdot (\varepsilon_{\eta qrs})^{2a\delta_{\eta,2}} \cdot \varepsilon_{2q}^{e}\cdot (-1)^{e}.
			\end{eqnarray*}
			
			Thus, $a +e\equiv 0\pmod2$.	As $\varepsilon_{2q}$ is not a square in $\QQ(\sqrt{2q},\sqrt{rs})$, we have $e=0$ and so $a=0$. Finally, $ {\varepsilon_{\rho qrs}}$ is  the only squarefree element in $E_{k_1}E_{k_2}E_{k_3}$ which is a square in $K_1$. Then 
			$$E_{K_1}=\langle -1, \varepsilon_2,\varepsilon_{rs}, \varepsilon_{\eta qrs}, \sqrt{\varepsilon_q},\sqrt{\varepsilon_{2q}}, \sqrt{\varepsilon_{\rho qrs}},\sqrt{\varepsilon_{2}\varepsilon_{rs}\varepsilon_{2rs}} \rangle.$$

			\item Assume that  $N(\varepsilon_{rs})=-1$ and $\sqrt{\varepsilon_2\varepsilon_{rs}\varepsilon_{2rs}}\notin k_2$.
			Then 
			\begin{eqnarray} E_{k_1}E_{k_2}E_{k_3}=\langle-1,  \varepsilon_{2}, \varepsilon_{rs},\varepsilon_{2rs},\varepsilon_{qrs},\varepsilon_{2qrs},\sqrt{\varepsilon_{q}},\sqrt{\varepsilon_{2q}  }\rangle.	\end{eqnarray}
			As above, we   consider $\chi$ an element of $ K_1$ of the form:
			$$\chi^2=\varepsilon_{2}^a\varepsilon_{rs}^b \varepsilon_{2rs}^c\varepsilon_{\eta qrs}^d\sqrt{\varepsilon_{q}}^e\sqrt{\varepsilon_{2q}}^f
			,$$
			where $a, b, c, d, e$ and $f$ are in $\{0, 1\}$.\\
			\noindent\ding{229} Let us start by   applying the norm $N_{K_1/\QQ(\sqrt{q},\sqrt{rs})}=1+\tau_1$. We have:
			\begin{eqnarray*}
				N_{K_1/\QQ(\sqrt{q},\sqrt{rs})}(\chi^2)&=& (-1)^a \cdot\varepsilon_{rs}^{2b}\cdot (-1)^{c}\cdot (\varepsilon_{\eta qrs})^{2d\delta_{\eta,1}}\cdot\varepsilon_{q}^{e}\cdot (-1)^{ e }\cdot  (-1)^{(\alpha+1)f}. 	\end{eqnarray*}
			Thus $a+c +e+(\alpha+1)f\equiv 0\pmod2$. As $\varepsilon_q$ is not a   square in $\QQ(\sqrt{q},\sqrt{rs})$, we get $e=0$. Therefore, $a+c  +(\alpha+1)f\equiv 0\pmod2$ and 
			$$\chi^2=\varepsilon_{2}^a\varepsilon_{rs}^b \varepsilon_{2rs}^c\varepsilon_{\eta qrs}^d\sqrt{\varepsilon_{2q}}^f
			.$$

			\noindent\ding{229} By   applying the norm $N_{K_1/\QQ(\sqrt{2q},\sqrt{rs})}=1+\tau_1\tau_2$, we get:
			\begin{eqnarray*}
				N_{K_1/\QQ(\sqrt{2q},\sqrt{rs})}(\chi^2)&=& (-1)^a \cdot\varepsilon_{rs}^{2b}\cdot (-1)^{c}\cdot (\varepsilon_{\eta qrs})^{2d\delta_{\eta,2}}\cdot \varepsilon_{2q}^{f}\cdot (-1)^{f }. 	\end{eqnarray*}
			Thus $f=0$. Therefore $a=c$ and we have
			$$\chi^2=\varepsilon_{2}^a\varepsilon_{rs}^b \varepsilon_{2rs}^a\varepsilon_{\eta qrs}^d .$$
			
			\noindent\ding{229} Now let us apply the norm $N_{K_1/\QQ(\sqrt{2},\sqrt{q})}=1+\tau_3$. We have:
			\begin{eqnarray*}
				N_{K_1/\QQ(\sqrt{2},\sqrt{q})}(\chi^2)&=& \varepsilon_2^{2a}\cdot (-1)^{b}\cdot (-1)^{a}. 	\end{eqnarray*}
			Thus $b=a$  and so we   have: 
			$$\chi^2=(\varepsilon_{2}\varepsilon_{rs} \varepsilon_{2rs})^a\varepsilon_{\eta qrs}^d
			.$$
			
		Since	$\sqrt{\varepsilon_{2}\varepsilon_{rs}\varepsilon_{2rs}}\not\in k_2$,  the case      $(F2)$ in Remark \ref{calcul} implies that 
			  $\sqrt{\varepsilon_{2}\varepsilon_{rs}\varepsilon_{2rs}}\notin K_1$ and $\sqrt{\varepsilon_{\eta qrs}}\sqrt{\varepsilon_{2}\varepsilon_{rs}\varepsilon_{2rs}}\in K_1$. Hence, 
			$$E_{K_1}=\langle -1, \varepsilon_2,\varepsilon_{rs}, \varepsilon_{\eta qrs}, \sqrt{\varepsilon_q},\sqrt{\varepsilon_{2q}}, \sqrt{\varepsilon_{\rho qrs}},\sqrt{\varepsilon_{2}\varepsilon_{rs}\varepsilon_{2rs}\varepsilon_{\eta qrs}} \rangle.$$  
		\end{enumerate}
		This completes the proof.
	\end{proof}

		\begin{corollary}\label{etacorolaK}
		Let $ q\equiv     3\pmod 4$ and  $r\equiv s\equiv    5\pmod 8$   be  distinct  prime numbers such that $\left(\frac{q}{r}\right)=\left(\frac{q}{s}\right)= (-1)^{\delta_{\eta,2}}$ and
	 	$\left(\frac{r}{s}\right)=1$. 	Put   $K=\QQ(\sqrt{\eta q},\sqrt{rs})$ and   $ \varepsilon_{\eta qrs}=\gamma+\gamma'\sqrt{\eta qrs}$. Then,  $  h_2(K)   =\frac{1}{2} h_2(rs)h_2(\eta qrs) $ or $\frac{1}{4} h_2(rs)h_2(\eta qrs)$. Moreover, we have:

		\begin{enumerate}[$1)$]
			\item     $   h_2(K) =\frac{1}{2} h_2(rs)h_2(\eta qrs) $ if and only if 
			$N(\varepsilon_{rs})=1$. 
			
			\item Assuming   
			that  $2^{\delta_{\eta,1}}q(\gamma-1)$   is not a square in $\NN$ and $N(\varepsilon_{rs})=1$, we get: 
			$$ |A(K_\infty)|=| A (K)| =\frac{1}{2} h_2(rs)h_2(\eta qrs).$$
		\end{enumerate}
	\end{corollary}
	\begin{proof}Let $K=\QQ(\sqrt{\eta q},\sqrt{rs})$, where $ q\equiv     3\pmod 4$ and  $r\equiv s\equiv    5\pmod 8$  are distinct  prime numbers such that $\left(\frac{q}{r}\right)=\left(\frac{q}{s}\right)= (-1)^{\delta_{\eta,2}}$ and 
		$\left(\frac{r}{s}\right)=1$.  As  $h_2(\eta q)=1$ (cf. \cite[Corollary 18.4]{connor88}), we have
		$$h_2(K)=\frac{1}{4}q(K) h_2(rs)h_2(\eta qrs).$$
	
	Notice that if  $N(\varepsilon_{rs})=1$, then, as mentioned in the proof of   Lemma \ref{lemmunitstriquad},  
	$\sqrt{\varepsilon_{rs}}=\alpha_1\sqrt{s}+\alpha_2\sqrt{r}	$ for some integers $\alpha_1$ and $\alpha_2$.
	So by	using  Lemma \ref{etaqrs}-$1)$, we  get	$E_{K}=\langle -1, \varepsilon_{\eta q},\varepsilon_{rs},   \sqrt{\varepsilon_{\eta qrs}} \text{ or }    \sqrt{\varepsilon_{rs}\varepsilon_{\eta qrs}} \rangle$ 
	according to whether $2^{\delta_{\eta,1}}q(\gamma-1) \in \NN^2$ or not. Thus,  $q(K)=2$.

	On the other hand, if $N(\varepsilon_{rs})=-1$, then by	using  Lemma \ref{etaqrs}-$1)$, we  get	$E_{K}=\langle -1, \varepsilon_{\eta q},\varepsilon_{rs},    \varepsilon_{\eta qrs}  \rangle$. 
	Thus,  $q(K)=1$. So we have the first part of the result.

	Now	assume   that $2^{\delta_{\eta,1}}q(\gamma-1)$   is not a square in $\NN$ and $N(\varepsilon_{rs})=1$. By class number formula, applied to $K_1$,    we get:
		\begin{eqnarray*} 
			h_2(K_1)&=&\frac{1}{2^{9}}q(K_1)  h_2(q) h_2(2q )h_2(rs) h_2(2rs)h(2)  h_2(qrs) h_2(2qrs)= \frac{1}{2}h_2(rs)h_2(\eta qrs) \nonumber 
		\end{eqnarray*}
		In fact,   $q(K_1)=2^4$ (cf.  Lemma \ref{lemmunitstriquad}), $h_2(q) =h_2(2q )=h(2) =1$ (cf. \cite[Corollary 18.4]{connor88}), $h_2(2rs)=4 $ (cf. \cite[Proposition $A_2$, p. 330]{kaplan76}) and $h_2(qrs)h_2({2qrs})=4h_2({ \eta qrs})$ (cf. the proof of Corollary \ref{corllaryh(F)=h(qrs)}).
		
		Since $K_1/K$ is ramified,   Fukuda's theorem implies that   $h_2(K_n)=\frac{1}{2} h_2(rs)h_2(\eta qrs)$ for all $n\geq 0$. This gives the result.
	\end{proof}

	\bigskip
	
	Notice that the condition  $N(\varepsilon_{rs})=1$    implies that  $\left(\frac{r}{s}\right)_4\not=\left(\frac{s}{r}\right)_4$ or 
	$\left(\frac{r}{s}\right)_4 =\left(\frac{s}{r}\right)_4=1$. Moreover, the condition $\left(\frac{r}{s}\right)_4\not=\left(\frac{s}{r}\right)_4$ implies that $N(\varepsilon_{rs})=1$ (cf. \cite[Theorem 1]{kuvcera1995parity}).
	We have:
	
	\bigskip

	\begin{theorem}\label{etatheoremmetacyclic-nonmodular}
		Let $ q\equiv     3\pmod 4$ and  $r\equiv s\equiv    5\pmod 8$   be  distinct   prime numbers such that $\left(\frac{q}{r}\right)=\left(\frac{q}{s}\right)= (-1)^{\delta_{\eta,2}}$, 
		$\left(\frac{r}{s}\right)=1$    and  $N(\varepsilon_{rs})=1$.  Put $F=\QQ(\sqrt{\eta qrs})$ and 	$ \varepsilon_{\eta qrs}=\gamma+\gamma'\sqrt{\eta qrs}$.
		Assume  that $2^{\delta_{\eta,1}}q(\gamma-1)$    is not a square  in $\NN$. 
		Then, we have: $$A(F)\simeq A(F_\infty)\simeq \ZZ/2\ZZ\times \ZZ/2^m\ZZ,$$
		where $m$  is such that $h_2(\eta qrs)=2^{m+1}$. Furthermore, the Galois group
		$\mathrm{Gal}(\mathcal{L}(F_\infty)/F_\infty)$,  where $\mathcal{L}(F_\infty)$ denotes the maximal unramified pro-$2$-extension  of $F_\infty$, is   
		a   {\bf metacyclic-nonmodular pro-$2$-group of Type 1 with $\alpha=2$} if and only if $\left(\frac{r}{s}\right)_4\not=\left(\frac{s}{r}\right)_4$.
		
	\end{theorem}
	\begin{proof}Put $K=\QQ(\sqrt{\eta q},\sqrt{rs})$, $K'=\QQ(\sqrt{r},\sqrt{\eta qs})$ and $K''=\QQ(\sqrt{s},\sqrt{\eta qr})$. Let $F_n$ (resp. $K_n$, $K_n'$ and $K_n''$)
		be the $n$th layer of the cyclotomic $\ZZ_2$-extension of $F$ (resp. $K$, $K'$ and $K''$).
		Notice that $h_2(\eta qrs)=h_2(F_1)=2^{m+1}$ (cf. Corollary \ref{corllaryh(F)=h(qrs)}) and so by Fukuda's theorem $ h_2(F_n)= 2^{m+1}$ for all $n\geq 0$.

		Let $n\geq 0$. By \cite[Theorem 5.9]{AziziRezzouguiZekhniniPeriodica} the $4$-rank of $F_n$ equals $1$.  So
		$A(F_n)\simeq  \ZZ/2\ZZ\times \ZZ/2^m\ZZ$.  
		Note that $K_n$, $K_n'$ and $K_n''$ are the three unramified extensions of $F_n$. Since by \cite[Lemma 5.5]{AziziRezzouguiZekhniniPeriodica}  $F_{3,2}=K$, this implies that
		$K_n$ is the unramified quadratic extension of $F_n$, for which all the three unramified biquadratic extensions over $F_n$, in its Hilbert $2$-class field, contain $K_n$. This means that we have the following diagram:
		\begin{figure}[H]
			\rotatebox{-00}{	$$
				\xymatrix{
					& F_n^{(2)} \ar@{<-}[d] & \\
					& F_n^{(1)}\ar@{<-}[ld]\ar@{<-}[d]\ar@{<-}[rd]\\
					{\bullet}\ar@{<-}[rd]&\ar@{<-}[ld] {\bullet}\ar@{<-}[d]\ar@{<-}[rd]  & {\bullet}\ar@{<-}[ld]\\
					{K_n'}\ar@{<-}[rd]& {K_n } \ar@{<-}[d]  & {K_n''}\ar@{<-}[ld]\\
					&F_n 
				}
				$$}
			\caption{}\label{fig2}
		\end{figure}

		Here the bullets (in Figure \ref{fig2}) are the unramified quartic extensions of $F_n$ within its Hilbert $2$-class field.

		We note that   we have  $\rg({A(K')}) = 2$ (cf. \cite[Théoème 1-$1)$]{AM2001}).
		So the fact that $K'_n/K'$ is ramified  gives $\rg({A(K'_n)}) \geq 2$.
		As under our conditions, $r$ and $s$ play symmetric roles,  we have $\rg({A(K''_n)}) \geq 2$. Therefore, by Lemma   \ref{metacyclic-nonmodular}-$1)$, $\rg({A(K'_n)})=\rg({A(K''_n)})=2 $.
		As  $\2r(A(K))=\2r(A(K_\infty))=2$ (cf. 	\cite[The main theorem, Items $5)$ and $14)$]{ChemseddinGreenbergConjectureI}), it follows that $\mathrm{Gal}(\mathcal{L}(F_\infty)/F_\infty)$ is a metacyclic-nonmodular $2$-group (cf. Lemma \ref{metacyclic-nonmodular}-$2)$).
		\begin{enumerate}[$1)$]
			\item If $\left(\frac{r}{s}\right)_4\not=\left(\frac{s}{r}\right)_4$,
			then $h_2(rs)=2$  (cf. \cite[Theorem 1]{kuvcera1995parity}). It follows from Corollary  \ref{etacorolaK}  that  $ h_2(K_n)=  h_2(F_n)= 2^{m+1}$. So by reading the classifications in Table \ref{tab1} (see the fifth column), we deduce that
			$$	A (K'_n)\simeq A (K''_n)\simeq \ZZ/2\ZZ\times \ZZ/2^m\ZZ  $$
			and 
			$$	 A (K_n)\simeq    {\ZZ/4\ZZ\times \ZZ/2^{m-1}\ZZ} .$$

			Moreover, $\mathrm{Gal}(\mathcal{L}(F_n)/F_n)$, the Galois group of the maximal unramified extension $\mathcal{L}(F_n)$ over $F_n$, is a  
	  metacyclic-nonmodular $2$-group of Type $1$ with $\alpha=2$. 
			
			\item If $\left(\frac{r}{s}\right)_4 =\left(\frac{s}{r}\right)_4=1$, then $h_2(rs) $ is divisible by $4$. So $h_2(rs)=2^{m'}$, with $m'\geq 2$ (cf. \cite[Theorem 1]{kuvcera1995parity}). 
			So by Corollary \ref{etacorolaK}, we have $  h_2(K_n)=2^{m+m'+1}   > 2^{m+1}=h_2(F_n)$. Therefore, by  Table \ref{tab1}, $\mathrm{Gal}(\mathcal{L}(F_n)/F_n)$ is not a  
			metacyclic-nonmodular $2$-group of Type $1$ with $\alpha=2$.
		\end{enumerate}Hence, the result follows  by taking the inverse limit.
	\end{proof}

	\begin{corollary}\label{corl1}		Let $ q\equiv     3\pmod 4$ and  $r\equiv s\equiv    5\pmod 8$   be distinct    prime numbers such that $\left(\frac{q}{r}\right)=\left(\frac{q}{s}\right)= (-1)^{\delta_{\eta,2}}$, 
		$\left(\frac{r}{s}\right)=1$ and 	$\left(\frac{r}{s}\right)_4\not=\left(\frac{s}{r}\right)_4$.   Put   	$ \varepsilon_{\eta qrs}=\gamma+\gamma'\sqrt{\eta qrs}$.
		Assume  that $2^{\delta_{\eta,1}}q(\gamma-1)$    is not a square  in $\NN$.
		Let $m$  be such that $h_2(\eta qrs)=2^{m+1}$.
		Then, we have:
		$$\mathrm{Gal}(\mathcal{L}(F_\infty)/F_\infty)=\langle a, b\ | \ a^{4}=1,  \  b^{2^m}=1,\  b^{-1}ab=a^{-1} \rangle  . $$
	\end{corollary}
	
\bigskip
	
	The     following corollary  gives some   families of real  biquadratic fields satisfying Greenberg's conjecture.

	 \medskip
	 \begin{corollary}\label{corol2}
	 	Let $ q\equiv     3\pmod 4$ and  $r\equiv s\equiv    5\pmod 8$   be  distinct   prime numbers such that $\left(\frac{q}{r}\right)=\left(\frac{q}{s}\right)= (-1)^{\delta_{\eta,2}}$,  
	 	$\left(\frac{r}{s}\right)=1$          and  $N(\varepsilon_{rs})=1$. Put $\varepsilon_{\eta qrs}=\gamma+\gamma'\sqrt{\eta qrs}$.	Assuming   that $2^{\delta_{\eta,1}}q(\gamma-1)$    is not a square  in $\NN$, we get the following items.
	 	\begin{enumerate}[$1)$]
	 		\item Let  $K'=\QQ(\sqrt{r},\sqrt{\eta qs})$, we have:
	 		$$ A(K'_\infty)\simeq A (K')\simeq \ZZ/2\ZZ\times \ZZ/2^m\ZZ.$$
	 		
	 		\item Let   $K =\QQ(\sqrt{\eta q},\sqrt{rs})$ and assume furthermore that $\left(\frac{r}{s}\right)_4\not=\left(\frac{s}{r}\right)_4$. We have:
	 		$$	A(K_\infty)\simeq A (K)\simeq {\ZZ/4\ZZ\times \ZZ/2^{m-1}\ZZ}.$$
	 	\end{enumerate}
	 	Here  $m$  is the integer such that $h_2(\eta qrs)=2^{m+1}$.
	 \end{corollary}	
	 \begin{proof}
	 	The second item is a direct deduction of the proof of Theorem \ref{etatheoremmetacyclic-nonmodular}. The idea of behind the first item   appears in the proof  Theorem \ref{etatheoremmetacyclic-nonmodular}, but for the reader's convenience   let us put   $K=\QQ(\sqrt{\eta q},\sqrt{rs})$, $K'=\QQ(\sqrt{r},\sqrt{\eta qs})$ and $K''=\QQ(\sqrt{s},\sqrt{\eta qr})$  and    $F =\QQ(\sqrt{\eta qrs})$.
	 	Notice that $A(F_n)\simeq \ZZ/2\ZZ\times \ZZ/2^m\ZZ$, in fact the $4$-rank of $F_n$ equals $1$ (cf. \cite[Theorem 5.9]{AziziRezzouguiZekhniniPeriodica}).
	 	As in the   proof of Theorem \ref{etatheoremmetacyclic-nonmodular}, we check that,  for all $n\geq 0$, we have $\2r(A(K'_n))\geq 2$.  Since $K_n'$ is an unramified quadratic extension of $F_n$ that is not contained in the three unramified quartic extention of $F_n$, we get $\2r(A(K'_n))= 2$ (cf.  Lemma   \ref{metacyclic-nonmodular}-$1)$). We deduce that $\2r(A(K'_n))= \2r(A(K''_n))=2$.
	 	Furthermore, as  $\2r(A(K_n))= 2$ (cf.  \cite[The main theorem, Items $5)$ and $14)$]{ChemseddinGreenbergConjectureI}), the Galois group   $\mathrm{Gal}(\mathcal{L}(F_\infty)/F_\infty)$ is a metacyclic-nonmodular  $2$-group (cf. Lemma   \ref{metacyclic-nonmodular}-$2)$). So we have the result by Table \ref{tab1}.   
	 \end{proof}

	 \medskip

		\begin{remark}\label{remexamp}
		Assume that $\eta$,  $q$, $r$ and $s$ satisfy the hypothesis of    Theorem \ref{etatheoremmetacyclic-nonmodular}.
		\begin{enumerate}[$1)$]
			\item
			If $h_2(\eta qrs)=2^{m+1}$, with $m\geq 4$, then the second item of Corollary \ref{corol2}  gives real biquadratic fields whose  $2$-Iwasawa modules are of $4$-rank equals $2$ and of $8$-rank equals $1$. 
			
			\item 
			Using 
			PARI/GP (cf. \cite{PARI}), we  get the following examples illustrating Corollary \ref{corol2}  and Theorem \ref{etatheoremmetacyclic-nonmodular} (cf. Tables \ref{tableexamples} and \ref{tableexamples2}):
			\begin{table}[H]
				{ \footnotesize
					$$ \begin{tabular}{  |c |c |c| c |c| c|}
						\hline	\rsp  $(\eta, q,r,s)$  & $h_2(\eta qrs)$ &$A(F)$ $(\simeq A(F_\infty))$ & $A(K)$ $(\simeq A(K_\infty))$ & $A(K')$ $(\simeq A(K'_\infty))$     \\ 
						\hline

						\rsp	$(1, 23,13,29)$      &    $2^3$&$\ZZ/2\ZZ\times \ZZ/4\ZZ  $ & $\ZZ/4\ZZ\times \ZZ/2\ZZ$  & $\ZZ/2\ZZ\times \ZZ/4\ZZ  $   \\
						\hline
						\rsp $(1, 107,13,29)$       &$2^4  $& $\ZZ/2\ZZ\times \ZZ/8\ZZ  $ & $\ZZ/4\ZZ\times \ZZ/4\ZZ$  &$\ZZ/2\ZZ\times \ZZ/8\ZZ  $       \\
						\hline 
						\rsp $(1, 443,181,13)$      &  $2^5  $&  $\ZZ/2\ZZ\times \ZZ/16\ZZ  $ &  $\ZZ/4\ZZ\times \ZZ/8\ZZ$   &   $\ZZ/2\ZZ\times \ZZ/16\ZZ  $      \\
						\hline

						\rsp $(2, 307,13,29)$      &$2^3       $ & $\ZZ/2\ZZ\times \ZZ/4\ZZ  $& $\ZZ/4\ZZ\times \ZZ/2\ZZ$  & $\ZZ/2\ZZ\times \ZZ/4\ZZ  $    \\ 
						\hline
						\rsp $(2, 31,13,29)$        &$2^4  $ &$\ZZ/2\ZZ\times \ZZ/8\ZZ  $  & $\ZZ/4\ZZ\times \ZZ/4\ZZ$  &$\ZZ/2\ZZ\times \ZZ/8\ZZ  $   \\ 
						\hline
						\rsp	$(2, 23,37,53)$   &  $2^5  $ &$\ZZ/2\ZZ\times \ZZ/16\ZZ  $&  $\ZZ/4\ZZ\times \ZZ/8\ZZ$   &   $\ZZ/2\ZZ\times \ZZ/16\ZZ  $     \\
						\hline
						
					\end{tabular}$$}
				\caption{Numerical examples with $r$ and $s$ satisfy $\left(\frac{r}{s}\right)_4\not=\left(\frac{s}{r}\right)_4$.} \label{tableexamples}
			 	\end{table}
		 	\begin{table}[H]
				{ \footnotesize
					$$ \begin{tabular}{  |c |c |c| p{2.4cm} |c| c|}
						\hline	\rsp  $(\eta, q,r,s)$  & $h_2(\eta qrs)$ &$A(F)$ $(\simeq A(F_\infty))$ &  \hspace*{0.43cm} $  A(K)$  \newline   $\quad \hspace*{0.43cm}A(K_1)$  & $A(K')$ $(\simeq A(K'_\infty))$     \\ 
						\hline

						\rsp 	$(1, 131,53,13)$   &  $2^3  $ &$\ZZ/2\ZZ\times \ZZ/4\ZZ  $&  $\ZZ/2\ZZ\times \ZZ/8\ZZ$\newline  $\ZZ/2\ZZ\times \ZZ/8\ZZ$ &   $\ZZ/2\ZZ\times \ZZ/4\ZZ  $     \\
						\hline                    
						
						\rsp 	$(1, 367,61,13)$   &  $2^4  $ &$\ZZ/2\ZZ\times \ZZ/8\ZZ  $&  $\ZZ/4\ZZ\times \ZZ/8\ZZ$\newline  $\ZZ/4\ZZ\times \ZZ/8\ZZ$ &   $\ZZ/2\ZZ\times \ZZ/8\ZZ  $     \\
						\hline 
						
						\rsp 	$(1, 491,61,13)$   &  $2^5  $ &$\ZZ/2\ZZ\times \ZZ/16\ZZ  $&  $\ZZ/8\ZZ\times \ZZ/8\ZZ$\newline   $\ZZ/8\ZZ\times \ZZ/8\ZZ$  &   $\ZZ/2\ZZ\times \ZZ/16\ZZ  $     \\
						\hline

						\rsp 	$(2, 67,53,13)$   &  $2^3  $ &$\ZZ/2\ZZ\times \ZZ/4\ZZ  $&  $\ZZ/2\ZZ\times \ZZ/8\ZZ$\newline  $\ZZ/2\ZZ\times \ZZ/8\ZZ$ &   $\ZZ/2\ZZ\times \ZZ/4\ZZ  $     \\
						\hline
						
						\rsp 	$(2, 19,53,13)$   &  $2^4  $ &$\ZZ/2\ZZ\times \ZZ/8\ZZ  $&  $\ZZ/4\ZZ\times \ZZ/8\ZZ$\newline  $\ZZ/4\ZZ\times \ZZ/8\ZZ$  &   $\ZZ/2\ZZ\times \ZZ/8\ZZ  $     \\
						\hline
						
						\rsp 	$(2, 163,181,29)$   &  $2^4  $ &$\ZZ/2\ZZ\times \ZZ/8\ZZ  $&  $\ZZ/4\ZZ\times \ZZ/16\ZZ$\newline  $\ZZ/4\ZZ\times \ZZ/16\ZZ$  &   $\ZZ/2\ZZ\times \ZZ/8\ZZ  $     \\
						\hline
						
						\rsp	$(2, 251,157,37)$   &  $2^5  $ &$\ZZ/2\ZZ\times \ZZ/16\ZZ  $&  $\ZZ/8\ZZ\times \ZZ/8\ZZ$  \newline $\ZZ/8\ZZ\times \ZZ/8\ZZ$    &   $\ZZ/2\ZZ\times \ZZ/16\ZZ  $     \\
						\hline
					\end{tabular}$$}
				\caption{Numerical examples with $r$ and $s$ satisfy $\left(\frac{r}{s}\right)_4 =\left(\frac{s}{r}\right)_4=1$.} \label{tableexamples2}
					\end{table}
		 	The primes $q$, $r$ and $s$ in Tables \ref{tableexamples} and \ref{tableexamples2} satisfy the hypothesis of   Theorem \ref{etatheoremmetacyclic-nonmodular}. More precisely, we have:
			\begin{enumerate}[\indent$\bullet$]
				\item The primes  $r$ and $s$ in   Table \ref{tableexamples} satisfy the condition $\left(\frac{r}{s}\right)_4\not=\left(\frac{s}{r}\right)_4$. So  for these primes, the Galois group $\mathrm{Gal}(\mathcal{L}(F_\infty)/F_\infty)$     is   
				a   { metacyclic-nonmodular pro-$2$-group of Type 1 with $\alpha=2$} (cf. Theorem \ref{etatheoremmetacyclic-nonmodular}).

				\item  The primes  $r$ and $s$ in   Table \ref{tableexamples2} satisfy the condition $\left(\frac{r}{s}\right)_4=\left(\frac{s}{r}\right)_4=1$. Hence for these primes, the Galois group $\mathrm{Gal}(\mathcal{L}(F_\infty)/F_\infty)$     is  {\bf not} 
				a   { metacyclic-nonmodular pro-$2$-group of Type 1 with $\alpha=2$} (cf. Theorem \ref{etatheoremmetacyclic-nonmodular}). 
			Moreover, we note that Corollary \ref{etacorolaK}  implies that   $ |A(K_\infty)|=  |A(K)|$ and since $K_n/K$ is totally ramified, the norm map $A(K_n)\rightarrow A(K)$ is  surjective. So $A(K_\infty)$ is isomorphic to $A(K)$. This fact is illustrated in the fourth column of Table \ref{tableexamples2}.	
			\end{enumerate}
		\end{enumerate}
		
	\end{remark}

	\begin{corollary}\label{corolminimalmetacyclic-nonmodular}
		
		Let $q$, $r$ and $s$ be prime numbers satisfying the hypothesis of Theorem \ref{etatheoremmetacyclic-nonmodular}.
		Assume that $h_2(\eta qrs)=8$.   Then      $$A(F)\simeq A(F_\infty)\simeq \ZZ/2\ZZ\times \ZZ/4\ZZ.$$
		Furthermore,  the Galois group 
		$\mathrm{Gal}(\mathcal{L}(F_\infty)/F_\infty)$, where  $\mathcal{L}(F_\infty)$ denotes the  unramified pro-$2$-extension  of $F_\infty$,
	 is   	a {\bf minimal  metacyclic-nonmodular} $2$-group of Type $1$ and of order $16$.
			\end{corollary}	
	\begin{proof}
		This is a direct deduction of Theorems \ref{etatheoremmetacyclic-nonmodular} and   \ref{AabounePrzekhiniNonmodular}.
	\end{proof}
	
	\medskip
	
We close the paper with the following result concerning real triquadratic fields.
	
		\medskip

	\begin{theorem}\label{corllarealtri(2 2)}
		Let $ q\equiv     3\pmod 4$ and  $r\equiv s\equiv    5\pmod 8$   be  distinct   prime numbers such that $\left(\frac{q}{r}\right)=\left(\frac{q}{s}\right)= (-1)^{\delta_{\eta,2}}$, 
		$\left(\frac{r}{s}\right)=1$ and 	$\left(\frac{r}{s}\right)_4\not=\left(\frac{s}{r}\right)_4$. 
		Assume   that $2^{\delta_{\eta,1}}q(\gamma-1)$    is not a square  in $\NN$ and let $m$ be the integer such that $h_2(\eta qrs)=2^{m+1}$.
		Then, for $\mathbb{F}=\QQ(\sqrt{\eta q},\sqrt{r},\sqrt{s})$, we have:
		$$  A(\mathbb{F})\simeq A(\mathbb{F}_\infty)\simeq \ZZ/2\ZZ\times \ZZ/2^{m-1}\ZZ.$$
		Moreover, $\mathrm{Gal}(\mathcal{L}(\mathbb{F}_\infty)/\mathbb{F}_\infty)$, the Galois group of the maximal unramified pro-$2$-extension $\mathcal{L}(\mathbb{F}_\infty)$ over $\mathbb{F}_\infty$, is abelian.
	\end{theorem}
	\begin{proof}  Let  $F_n$ (resp. $\mathbb{F}_n$) be the $n$th layer of the cyclotomic $\ZZ_2$-extension of $F=\QQ(\sqrt{\eta qrs})$ (resp. $\mathbb{F}$). 
		According to  Theorem  \ref{etatheoremmetacyclic-nonmodular}, we have    $  A(F_n)\simeq \ZZ/2\ZZ\times \ZZ/2^m\ZZ $
		 and	$\mathrm{Gal}(\mathcal{L}(F_n)/F_n)$ is   
		a metacyclic-nonmodular $2$-group of Type 1 with $\alpha=2$.  Notice that  $\mathbb{F}_n$ is the abelian quartic unramified $2$-extension of  {$F_n$}, within its Hilbert $2$-class field, that contains all its three quadratic unramified extensions. Therefore, by Table \ref{tab2}, 
		$   A(\mathbb{F}_n)\simeq \ZZ/2\ZZ\times \ZZ/2^{m-1}\ZZ.$ Thus, the result follows  by taking the inverse limit.
	\end{proof}

	\section*{Acknowledgments}
We are grateful to the reviewer for the insightful comments, which have contributed to improving this paper.
	We would also like  to thank Brahim Aaboun for useful discussions on the content of his paper cited in references.

	 	\section*{Data availability}	All data generated or analyzed during this study are included in this article.


\begin{thebibliography}{11}
		
		\bibitem{aaboune} B. Aaboun and A. Zekhnini, 	On the metacyclic 2-groups whose abelianizations are of type $(2,2^n)$, $n\geq2$ and applications, Res. Number Theory {9}, 55 (2023)
		
		
		
		\bibitem{Az-00} A. Azizi,   {Sur la capitulation des $2$-classes d'id\'eaux de	$\kk =\QQ(\sqrt{2pq}, i)$,  o\`u $p\equiv-q\equiv1\pmod4$, }{  Acta. Arith.  { 94} (2000),  383-399.} 
		
		\bibitem{AM2001} A. Azizi  et A. Mouhib, {Sur le rang du $2$-groupe de classes de $\mathbb{Q}( \sqrt{m},\sqrt{d} )$ où $m=2$ ou un premier $p \equiv 1 \pmod{4}$,} {Trans. Amer. Math. Soc., 353 (2001),  2741–2752.}
		
		\bibitem{AziziRezzouguiZekhniniPeriodica} A. Azizi,  M.  Rezzougui and   A.   Zekhnini,   On the maximal unramified pro-$2$-extension of certain cyclotomic $\mathbb Z_2$-extensions, Period. Math. Hung. 83 (2021), 54–66.
		
		\bibitem{AziziRezzouguiZekhniniDebrecen} A. Azizi,  M.  Rezzougui and    A. Zekhnini,   Structure of the Galois group of the maximal unramified pro-$2$-extension of some $\mathbb{Z}_2$-extensions, 100 (2022), 11-28.
		
		
		
		
		\bibitem{AzZektaous}	A. Azizi,     A. Zekhnini     and M. Taous,  On the strongly ambiguous classes of $\mathbb{K}/\QQ(i)$ where $\mathbb{K}=\QQ(\sqrt{2p_1p_2},i)$,   Asian-European Journal of Mathematics, 7  (2014), 1450021 (26 pages)
		
		
			\bibitem{AziTaouss}	 	A. Azizi    and M. Taous, 	Determinations des corps $K=\QQ(\sqrt{d},\sqrt{-1})$ dont les 2-groupes de classes sont de type (2,4) ou (2,2,2), Rend. Inst. Mat. Univ. Trieste, Vol. XL, 93-116 (2009).
		
			\bibitem{Ben17} E. Benjamin, Some real quadratic number fields with their Hilbert 2-class field having cyclic $2$-class group,  J. Number Theory, 173 (2017),  529-546.
		
		\bibitem{Ben2006} E. Benjamin, On the 2-class field tower of some imaginary biquadratic number fields,	Ramanujan J.,   11 (2006), 103–110.
		
	\bibitem{Ben1999}	E. Benjamin, On the Second Hilbert $2$-Class Field of Real Quadratic Number Fields with 2-Class Group Isomorphic to $(2,2^n)$, $n\geq2$,  Rocky Mountain J. Math., (1999), 763-788.  
		

		
		\bibitem{benjashnepreprint1993} E. Benjamin and C.   Snyder,  Number fields with 2-class group number isomorphic to $(2,2^m)$, preprint. This reference cannot be shared openly, to protect the rights of its authors. It may be requested from its authors (1993).
		
		\bibitem{BenjShn(22)} 	E. Benjamin and	 	C. Snyder,	 Real quadratic number fields with   $2$-class group of type $(2,2)$,   Math. Scand 76 (1995), 161-178.
		
		
		\bibitem{chemseddineArith} M. M. Chems-Eddin, Arithmetic of some real triquadratic fields; Units and 2-class groups,
		Moroc. J. Algebra Geom. Appl., 3 (2024), 358-387.
		
		
		\bibitem{ChemseddinGreenbergConjectureI} M. M. Chems-Eddin, Greenberg's conjecture and Iwasawa module of real biquadratic fields I, J. Number Theory, 281 (2026),   224-266
	  
		
			\bibitem{chems24} M.M. Chems-Eddin, {On the maximal unramified pro-2-extension of	$\ZZ_2$-extension of certain real biquadratic fields}, preprint (2024), http://arxiv.org/abs/2409.13574v1
	 	
		
		\bibitem{chemsmohahajjami} M. M. Chems-Eddin, M. B. T. El Hamam, and M. A. Hajjami, On the unit group and the 2-class number of $\QQ(\sqrt{2},\sqrt{p},\sqrt{q})$, Ramanujan J., 65 (2024), 1475–1510.
		
	
		
	 
		
		
		
		
		
		\bibitem{connor88}	P. E. Conner,  and J. Hurrelbrink,    {Class number parity, }{  Ser. Pure. Math. $8$,  World Scientific,  1988. }
		
		
		
		
		
		\bibitem{fukuda}  T. Fukuda,     \textit{Remarks on $\mathbb{Z}_p$-extensions of number fields,}{ Proc. Japan Acad. Ser. A Math. Sci. { 70}  $(1994)$,  $264$–$266$.}
		
		
		
		
		
		
		
		
		
		
		
		
		
		
		
		\bibitem{He-22} G. Herglotz,  {\it \"Uber einen Dirichletschen Satz,}{ Math. Z.,   12    (1922),  255--261.}
		
		
		
		
		
		\bibitem{kaplan76}	P. Kaplan,    {Sur le $2$-groupe des classes d'id\'eaux des corps quadratiques,}{ J. Reine. Angew. Math.,  283/284 (1976),  313-363.}
		
		\bibitem{Ki76}	H.  Kisilevsky,  {Number fields with class number congruent to $4$ mod $8$ and Hilbert's
			Theorem $94$,}{ J. Number Theory, {8} (1976), 271-279.} 
		

		
		
		
		
		
		\bibitem{kuvcera1995parity}	R. Ku{\v{c}}era,    On the parity of the class number of a biquadratic field, {J. Number Theory}, {52} (1995),  43–52.
		
		
		
		
		
		
		
		
		
		
		
		
		
		\bibitem{Ku-50} S. Kuroda,  {\"Uber die Klassenzahlen algebraischer Zahlk\"orper,}{ Nagoya Math. J., 1 (1950),  1--10.}
		
		
		
		
		
		\bibitem{LaxmiSaikia} H.	Laxmi,  A. Saikia,  $\mathbb{Z}$-extension of real quadratic fields with $ \mathbb{Z} /2\mathbb{Z}$
		as $2$-class group at each layer, Ramanujan J 64, 1285–1301 (2024). https://doi.org/10.1007/s11139-024-00869-8
		

		
		\bibitem{MizusawaAnnMath2014} Y. Mizusawa,  A note on semidihedral 2-class field towers and $\mathbb Z_2$-extensions, Ann. Math. Qué. 38 (2014),  73–79.
		
		\bibitem{MizusawaThesis}Y. Mizusawa, A Study of Iwasawa Theory on Class Field Towers, A dissertation submitted for the degree of Doctor of Science at Waseda University, 2004.
		
		\bibitem{Ajn}  A. Mouhib, {On the parity of the class number of multiquadratic number	fields.} J.   Number Theory,  129 (2009) 1205–1211	
		
		\bibitem{MohibMathNachr2016}  A. Mouhib, Capitulation of the 2-class group of some cyclic number fields with large degree, Math. Nachr. 289 (2016) 1927-1933.
		
		

		\bibitem{PARI}	The PARI Group, PARI/GP version 2.17.3 (64 bit), Univ. Bordeaux, 2025,
		 \href{https://pari.math.u-bordeaux.fr/}{https://pari.math.u-bordeaux.fr/}
		
		
		\bibitem{Wa-66} H. Wada,   {On the class number and the unit group of certain algebraic number fields.}{ J. Fac. Univ. Tokyo.   13  (1966), 201-209.}
		
		\bibitem{washington1997introduction}L.C. Washington, { Introduction to cyclotomic fields.}{ Graduate Texts in Mathematics, 83. Springer-Verlag, New York, 1982.}
		
		
		
		
		
	\end{thebibliography}
\end{document}